\documentclass[11pt,reqno]{amsart}

\usepackage[english]{babel}
\usepackage[utf8]{inputenc}
\usepackage[T1]{fontenc}
\usepackage{lmodern} \normalfont %to load T1lmr.fd 
\DeclareFontShape{T1}{lmr}{bx}{sc} { <-> ssub * cmr/bx/sc }{}
\usepackage{microtype}	% for improved spacing between words and letters

\usepackage{comment}

\usepackage{amssymb}
\usepackage{amsmath}
\usepackage{amsthm}
\usepackage{mathtools}
\usepackage{mathrsfs}
\usepackage{accents} % correct spacing for accents in sub- and superscript
\usepackage{etoolbox}
\usepackage{siunitx}
\usepackage{dsfont} % for special 1, \amthds{1}
\usepackage{graphicx}
\usepackage{color}
\usepackage[dvipsnames]{xcolor}
\usepackage{tikz}
\usetikzlibrary{calc,positioning,shapes}
\usetikzlibrary{patterns,decorations.pathmorphing,decorations.markings}
\usepackage{pgfplots}
\pgfplotsset{compat=newest}
\usepackage[margin=10pt,font=small,labelfont=bf,labelsep=endash]{caption}
\usepackage{subcaption}

\usepackage{booktabs}
\usepackage{paralist}
\usepackage{soul}

\numberwithin{equation}{section}

\usepackage{algorithm}
\usepackage{algorithmicx}
\usepackage{algpseudocode}

\usepackage{cite}
\usepackage{multirow} % for multirows in tabular environment
\usepackage{enumitem} % for labelling items in enumerate
\setlist[enumerate]{label=(\roman*)}

\usepackage[colorlinks,citecolor=black,urlcolor=black]{hyperref}
	\hypersetup{allcolors=black}
\usepackage[nameinlink,noabbrev]{cleveref}

\theoremstyle{plain}
\newtheorem{theorem}{Theorem}[section]

\newtheorem{lemma}{Lemma}

\newtheorem{remark}{Remark}

\newtheorem{assumption}{Assumption}

\makeatletter
\let\c@proposition\c@theorem
\let\c@lemma\c@theorem
\let\c@corollary\c@theorem
\let\c@remark\c@theorem
\let\c@definition\c@theorem
\let\c@assumption\c@theorem
\let\c@example\c@theorem
\makeatother



\crefname{theorem}{theorem}{theorems}
\Crefname{theorem}{Theorem}{Theorems}
\crefname{proposition}{proposition}{propositions}
\Crefname{proposition}{Proposition}{Propositions}
\crefname{lemma}{lemma}{lemmas}
\Crefname{lemma}{Lemma}{Lemmas}
\crefname{corollary}{corollary}{corollaries}
\Crefname{corollary}{Corollary}{Corollaries}
\crefname{remark}{remark}{remarks}
\Crefname{remark}{Remark}{Remarks}
\crefname{definition}{definition}{definitions}
\Crefname{definition}{Definition}{Definitions}
\crefname{assumption}{assumption}{assumptions}
\Crefname{assumption}{Assumption}{Assumptions}
\crefname{example}{example}{examples}
\Crefname{example}{Example}{Examples}

\crefname{section}{section}{sections}
\Crefname{section}{Section}{Sections}
\crefname{subsection}{subsection}{subsections}
\Crefname{subsection}{Subsection}{Subsections}
\crefname{subsubsection}{subsubsection}{subsubsections}
\Crefname{subsubsection}{Subsubsection}{Subsubsections}

\newcommand{\J}{\mathcal{J}}
\newcommand{\R}{\mathbb{R}}

\newcommand{\N}{\mathbb{N}}

\DeclareMathOperator*{\argmin}{arg\,min} %argmax
\DeclareMathOperator{\Sym}{sym}

\DeclareMathOperator{\vspan}{span}
\DeclareMathOperator{\reshape}{reshape}

\newcommand{\tenp}[2]{\underset{#1,#2}{*}}

\newcommand{\norm}[1]{\left\| #1 \right\|}

\newcommand{\TenDeriv}[1]{\mathbf{D}^{(#1)}}

\definecolor{mycolor1}{rgb}{0.00000,0.44700,0.74100}% blue
\definecolor{mycolor2}{rgb}{0.85000,0.32500,0.09800}% red
\definecolor{mycolor3}{rgb}{0.92900,0.69400,0.12500}% orange/yellow
\definecolor{mycolor4}{rgb}{0.46600,0.67400,0.18800}% green
\definecolor{mycolor5}{rgb}{0.49400,0.18400,0.55600}% purple
\definecolor{mycolor6}{rgb}{0.00000,0.2,0.4}% dark blue

\newcommand{\Matlab}{MATLAB\textsuperscript{\textregistered}}

\title[Higher order extended Kalman filters]{Approximations of the Mortensen observer using higher order extended Kalman filters}
\author{Tobias Breiten${}^{\star}$ \and Justus Ramme${}^\star$ \and Jesper Schr\"oder${}^\dagger$}
\address{${}^{\star}$ Institute of Mathematics, Technische Universität Berlin, Straße des 17. Juni 136, 10623 Berlin, Germany}
\address{${}^{\dagger}$ Johann Radon Institute, Austrian Academy of Sciences, A-4040 Linz, Austria}
\email{tobias.breiten@tu-berlin.de}
\email{ramme@math.tu-berlin.de}
\email{jesper.schroeder@ricam.oeaw.ac.at}
\date{\today}

\begin{document}
%
%%%%%%%%%%%%%%%%%%%%%%%%%%%%%%%%%%%%%%%%%%%%%%%%%%%%%%%%%%%%%%%%%%%%%%%%%%%%%%%
%%%%%%%%%%%%%%%%%%%%%%%%%%%%%%%%%%%%%%%%%%%%%%%%%%%%%%%%%%%%%%%%%%%%%%%%%%%%%%%
\begin{abstract}
  A polynomial approximation of the minimum energy estimator, also called Mortensen observer, is discussed. The method relies on successive differentiations of an underlying value function and the Hamilton-Jacobi-Bellman equation, respectively. By means of neglecting higher order derivatives of the value function along the unknown observer trajectory, a coupled set of nonlinear tensor structured differential equations is derived. In its simplest form, the approach boils down to the well-known extended Kalman filter. Numerical experiments with polynomials up to the order eight illustrate the potential of the new approach and indicate local convergence to the Mortensen observer.
\end{abstract}
%%%%%%%%%%%%%%%%%%%%%%%%%%%%%%%%%%%%%%%%%%%%%%%%%%%%%%%%%%%%%%%%%%%%%%%%%%%%%%%
%%%%%%%%%%%%%%%%%%%%%%%%%%%%%%%%%%%%%%%%%%%%%%%%%%%%%%%%%%%%%%%%%%%%%%%%%%%%%%%
%

\maketitle


%15A69 - Multilinear algebra, tensor calculus
%49L12 - Hamilton-Jacobi equations in optimal control and differential games
%93B53 - Observers
%93C15 - Control/observation systems governed by ordinary differential equations

% 
%%%%%%%%%%%%%%%%%%%%%%%%%%%%%%%%%%%%%%%%%%%%%%%%%%%%%%%%%%%%%%%%%%%%%%%%%%%%%%%
%%%%%%%%%%%%%%%%%%%%%%%%%%%%%%%%%%%%%%%%%%%%%%%%%%%%%%%%%%%%%%%%%%%%%%%%%%%%%%%
\section{Introduction}
%%%%%%%%%%%%%%%%%%%%%%%%%%%%%%%%%%%%%%%%%%%%%%%%%%%%%%%%%%%%%%%%%%%%%%%%%%%%%%%
%%%%%%%%%%%%%%%%%%%%%%%%%%%%%%%%%%%%%%%%%%%%%%%%%%%%%%%%%%%%%%%%%%%%%%%%%%%%%%%
%

This work considers a nonlinear system of the form
\begin{equation}\label{eq: distSys}
	\begin{aligned}
		\dot{x}(t) &= f(x(t)) + G v(t),~ t \in (0,T],\\
		x(0) &= x_0 + \eta,
	\end{aligned} 
\end{equation}
where $f\colon \mathbb R^n\to \mathbb R^n, G\in \mathbb R^{n,m}$ and $x_0 \in \mathbb R^n$ are known characteristics of the dynamic and $v\in L^2(0,T;\mathbb R^m), \eta \in \mathbb R^n$ denote unknown deterministic disturbances. Given 
measurements 
\begin{equation}\label{eq: distMeas}
	y(t) = h(x(t)) + \mu(t)
\end{equation}
with $h\colon \mathbb R^n \to \mathbb R^p$ which are subject to deterministic disturbances $\mu \in L^2(0,T;\mathbb R^p),$ our goal is to derive a sequence of nonlinear equations which locally approximate the minimum energy estimator, also called Mortensen observer, originally proposed in \cite{Mor68}. Referring to the latter work, let us consider the Luenberger-type observer

\begin{equation}\label{eq: observerEq}
	\begin{aligned}
		\dot{\widehat{x}}_{\mathrm{M}}(t) &= f(\widehat{x}_{\mathrm{M}}(t)) 
		+  (\nabla^2 \mathcal{V}(t,\widehat{x}_{\mathrm{M}}(t)))^{-1} \mathcal{J}_h(\widehat{x}_{\mathrm{M}}(t))^\top Q (y(t) - h(\widehat{x}_{\mathrm{M}}(t) ) ),~ t \in (0,T],\\
		\widehat{x}_{\mathrm{M}}(0) &= x_0,
	\end{aligned}
\end{equation}
where $\nabla^2 \mathcal{V}$ and $\mathcal{J}_h$ denote the Hessian and the Jacobian with respect to the spatial variable, respectively.
The time and space dependent output injection operator is implicitly characterized by the following time-dependent, nonhomogeneous Hamilton-Jacobi-Bellman (HJB) equation
\begin{equation}\label{eq: HJB}
	\begin{aligned}
		\partial_t \mathcal{V}(t,\xi) &=
		-\nabla \mathcal{V}(t,\xi)^\top f(\xi) 
		- \frac{1}{2} \Vert G^\top \nabla \mathcal{V}(t,\xi) \Vert_{R^{-1}}^2
		+ \frac{1}{2} \Vert y(t) - h(\xi) \Vert_Q^2,\\
		\mathcal{V}(0,\xi) &= \frac{1}{2} \Vert \xi - x_0 \Vert_\Gamma^2,
	\end{aligned}
\end{equation}
where $\nabla \mathcal{V}$ denotes the spatial gradient.
Here, $\mathcal{V}\colon [0,T]\times \mathbb R^n \to \mathbb R$ is the minimal value function associated with the (backwards running) optimal control problem 
\begin{equation}\label{eq:P_t} 
\begin{aligned}
&\underset{v\in L^2(0,t;\mathbb R^m)}{\min} ~\mathcal{J}(t,\xi; v) = \frac{1}{2} \|x(0) - x_0 \|_\Gamma^2 + \frac{1}{2} \int_{0}^{t} \|v(s)\|_{R}^2 + \alpha \|y(s)-h(x(s))\|_{Q}^2 ~\mathrm{d}s \\[2.5ex]
&\text{s.t. } \dot{x}(s) = f(x(s)) + Gv(s),~s \in (0,t), ~x(t) = \xi, \\
\end{aligned}
\end{equation}
with symmetric, positive definite weighting matrices $\Gamma$, $R$, and $Q$.
Assuming sufficient smoothness of $\mathcal{V}$, in \cite{Mor68} and later also in \cite{Fle97} it was shown that if one defines an estimate as
\begin{equation}\label{eq: observerDef}
	\widehat{x}_{\mathrm{M}}(t) \coloneqq \argmin\limits_{\xi \in \mathbb{R}^n} \mathcal{V}(t,\xi),
\end{equation}
the dynamical system \eqref{eq: observerEq} characterizes the evolution of $\widehat{x}_{\mathrm{M}}(\cdot)$ over $[0,T]$. While there exist research works on the Mortensen observer taking a filtering perspective, see, e.g., \cite{Hij80,Hij82,ChaEtAl23,ChauMeMoZi25} and \cite{Moi18} for the discrete-time case, throughout this manuscript, we follow the deterministic viewpoint proposed in the original formulation \cite{Mor68}. Let us also refer to \cite{Fle97} and \cite{Wil04} where detailed derivations based on a deterministic optimal control framework have been presented. Due to the strong assumptions required for rigorously defining \eqref{eq: observerEq}, the literature focusing on the theoretical aspects of the Mortensen observer is rather scarce, often assuming the involved nonlinearities to be globally Lipschitz \cite{Kre03}. More recently, for systems with quadratic nonlinearity,  we have shown the local well-posedness of \eqref{eq: observerEq} by employing a sensitivity analysis for the value function in \eqref{eq: HJB}, see \cite{BreS24}. 
It is well known that for linear dynamics and measurements, i.e., $f(\xi) = A \xi $ and $h(\xi) = C \xi$, with matrices $A \in \mathbb{R}^{n,n}$, $C \in \mathbb{R}^{p,n}$ the Mortensen observer reduces to the classic Kalman filter.
Hence, the Mortensen observer can be viewed as an energy optimal extension of the Kalman filter to nonlinear systems. Unfortunately, its dependence on $\nabla^2 \mathcal{V}$ makes its numerical realization a highly challenging task. While in principal the value function is available as the solution of \eqref{eq: HJB}, solving such HJB equations is known to suffer from the curse of dimensionality \cite{Bel61}. 

Recently, various approaches tackling this issue were proposed. Let us mention the learning based methods utilizing neural networks  \cite{BreKu21} or structured polynomials \cite{BreKuSc23}. We also point to an algorithmic realization based on max-plus algebra \cite{FleMcE00}. Here we take a different route and instead consider numerical approaches based on the local approximation of the value function. 
Assuming sufficient regularity of $\mathcal{V}$, one can  characterize the Hessian along the observer $ \Pi(t) \coloneqq \nabla^2 \mathcal{V}(t,\widehat{x}_\mathrm{M}(t)) $ via the differential equation
\begin{equation}\label{eq: Hessian_V}
\begin{aligned}
    &\dot{\Pi}(t) =
    -\mathcal{J}_f(\widehat{x}_{\mathrm{M}}(t))^\top \Pi(t)
    -\Pi(t) \mathcal{J}_f(\widehat{x}_{\mathrm{M}}(t)) 
    - \Pi(t) G R^{-1} G^\top \Pi(t) 
    +
    \mathcal{J}_h(\widehat{x}_{\mathrm{M}}(t))^\top Q \mathcal{J}_h(\widehat{x}_{\mathrm{M}}(t)) \\
    &\qquad + \TenDeriv{3} \mathcal{V} (t,\widehat{x}_{\mathrm{M}}(t))  
    \tenp{1}{1}
    \Pi(t)^{-1} (\mathcal{J}_h(\widehat{x}_{\mathrm{M}}(t))^\top Q (y(t)-h(\widehat{x}_{\mathrm{M}}(t))) \\
    &\qquad - \TenDeriv{2} h (\widehat{x}_{\mathrm{M}}(t))  \tenp{1}{1} Q (y(t)-h(\widehat{x}_{\mathrm{M}}(t))), \\
    &\Pi(0) = \Gamma.
\end{aligned}
\end{equation}
Here $\TenDeriv{3} \mathcal{V}(t,\xi) \in \mathbb{R}^{n,n,n}$ and $ \TenDeriv{2} h(\xi) \in \mathbb{R}^{p,n,n}$ denote order three spatial derivatives as tensors and $\tenp{1}{1}$ denotes an appropriate tensor multiplication, cf.~\Cref{subsec:tensor_prelim} below for the details.
\Cref{eq: Hessian_V} is of interest from different perspectives. On the one hand, if we assume that $\TenDeriv{2} h$ and $\TenDeriv{3} \mathcal{V}$ vanish along $\widehat{x}_{\mathrm{M}}(\cdot),$ then \eqref{eq: observerEq} together with \eqref{eq: Hessian_V} boil down to the well-known extended Kalman filter
\begin{equation}\label{eq: ext Kalman}
	\begin{aligned}
		\dot{\widehat{x}}_{\mathrm{K}}(t) &= f(\widehat{x}_{\mathrm{K}}(t)) + \Sigma(t) \, \mathcal{J}_h(\widehat{x}_{\mathrm{K}}(t))^\top Q ( y(t) - h(\widehat{x}_{\mathrm{K}}(t))), \\ \widehat{x}_{\mathrm{K}}(t) &= x_0, \\[2ex]
        \dot{\Sigma}(t) &= \mathcal{J}_f(\widehat{x}_{\mathrm{K}}(t)) \, \Sigma(t) + \Sigma(t) \, \mathcal{J}_f(\widehat{x}_{\mathrm{K}}(t))^\top \\
        & \quad 
		- \Sigma(t) \, \mathcal{J}_h(\widehat{x}_{\mathrm{K}}(t))^\top Q (\mathcal{J}_h(\widehat{x}_{\mathrm{K}}(t))) \, \Sigma(t) + G R^{-1} G^\top, \\ \Sigma(0) &= \Gamma^{-1},
	\end{aligned}
\end{equation}
where we have replaced \eqref{eq: Hessian_V} by an explicit equation for $\Sigma(t)\coloneqq\nabla^2 \mathcal{V}(t,\widehat{x}_{\mathrm{K}}(t))^{-1}.$ On the other hand, it turns out that for non-vanishing $\TenDeriv{3} \mathcal{V}$, the second derivative involves the (unknown) third derivative of the value function. In fact, this analogously carries over to all higher order derivatives of $\mathcal{V}$ and was already referred to as \emph{closure problem} in \cite{Mor68}. While this naturally prevents from iteratively computing any of the derivatives of $\mathcal{V}$ ``exactly'', at the same time it allows for numerical approximations based on the assumption that $\TenDeriv{k}\mathcal{V}$ vanishes and it will be the starting point of our contribution here.
Note that the idea of approximating $\nabla^2 \mathcal{V}$ by means of a Taylor series is not new and has been proposed in \cite{Kre15,Kre18} for discrete-time systems. In fact, using power series based approximations was already suggested in \cite{Alb61} in the closely related context of optimal feedback stabilization. Here, the literature is quite rich and we exemplarily mention \cite{Luk69} for a theoretical analysis of Taylor series approximations for analytic control systems, \cite{AguK14,KreAH13} for a numerical realization for small and medium scale systems as well as the more recent work \cite{BorZ21,BreKP18,Kraetal24,CorK25} utilizing tensor calculus in the large scale case. Moreover, we further refer to \cite{BreKP19} where polynomial feedback laws have been studied in the context of infinite-dimensional bilinear systems. Despite the overall common goal of locally approximating the value function, it should be pointed out that there is a fundamental conceptual difference that distinguishes the computation of optimal feedback laws from the computation of the Mortensen observer.

In particular, the first class of methods typically considers infinite horizon stabilization problems around an a priori given steady state $x_s$, often $x_s=0$, and subsequently aims for an approximation $F$ of the optimal feedback taking the form
\begin{align}\label{eq:vf_approx_cont}
    F(x)
    =
    -G^\top \nabla \mathcal{V}_\mathrm{apr}(x)
    \approx
    -G^\top \nabla \mathcal{V}(x),
\end{align}
where the value function is locally approximated using its $k$-th Taylor polynomial, i.e.,
\begin{equation*}
\begin{aligned}
   \mathcal{V}_\mathrm{apr}(x)
    &=
    \mathcal{V}(x_s)
    +
    \nabla \mathcal{V}(x_s)^\top (x-x_s)
    +
    \frac{1}{2} (x-x_s)^\top \nabla^2 \mathcal{V}(x_s) (x-x_s)\\
    &+
    \sum_{i=3}^k \frac{1}{i!} \TenDeriv{k} \mathcal{V}(x_s) (x-x_s)^k,
\end{aligned}
\end{equation*}
where the last summand formally denotes the multiplication of $(x-x_s)$ to each of the $k$ sides of the tensor $\TenDeriv{k} \mathcal{V}(x_s)$.
These tensors are obtained as solutions of a set of tensor equations where the equation characterizing $\TenDeriv{j} \mathcal{V}$ depends only on $\TenDeriv{i} \mathcal{V}$, $i\leq j$. Hence they can be solved iteratively and yield exact derivatives. The obtained feedback law is expected to perform well, for initial states $x_0$ sufficiently close to $x_s$. Whenever the closed loop system driven by the feedback $F$ visits states $x$ far away from $x_s$, the approximation quality of the Taylor polynomial may degrade quickly.

The approximation scheme discussed in this work is conceptually different. In fact, assuming knowledge of the third derivative $\TenDeriv{3} \mathcal{V} (t,\widehat{x}_{\mathrm{M}}(t))$, \cref{eq: Hessian_V} yields the exact optimal observer gain and coupling it with the observer equation \eqref{eq: observerEq} would yield the optimal observer $\widehat{x}_\mathrm{M}$. The derivation of the tensor equation characterizing $\TenDeriv{3} \mathcal{V} (t,\widehat{x}_{\mathrm{M}}(t))$ however, corresponds to an expansion of $\mathcal{V}(t,\cdot)$ around $\widehat{x}_\mathrm{M}(t)$. The time-dependence of the expansion point introduces a coupling such that determining the third derivative requires the forth. As mentioned above, this effect carries over to all higher order derivatives. Hence truncating the $k$-th order derivative implicitly affects the approximation of $\TenDeriv{3} \mathcal{V} (t,\widehat{x}_{\mathrm{M}}(t))$ which enters the Riccati-like equation \eqref{eq: Hessian_V}. Therefore the obtained approximations of derivatives are not in fact coefficients of the associated Taylor polynomial and as of now no results regarding convergence or error bounds are available. On the other hand, the approach is not as sensitive towards the initial state. As pointed out above, an application of the feedback control \eqref{eq:vf_approx_cont} to a system with initial state $x_0$ far away from the equilibrium $x_s$ would almost certainly fail. The same effect does not occur in the scheme analyzed in this work, as the expansion of the value function is done around the minimizing trajectory $\widehat{x}_\mathrm{M}$ itself. 

Our main results now can be summarized as follows:
\begin{itemize}
    \item Utilizing a generalized product rule for the differentiation of two vector-valued functions, in \Cref{LEM:HigherDerivativeHJB} we derive explicit expressions for arbitrary higher order derivatives of the HJB equation \eqref{eq: HJB}.
    \item Assuming a particular order of the derivative of $\mathcal{V}$ to vanish along the Mortensen observer, in \Cref{thm:main} we characterize $\widehat{x}_\mathrm{M}$ by a set of nonlinear and partially structured tensor differential equations. For the general case without a vanishing derivative, the same set of equations subsequently defines an approximation to the Mortensen observer, thereby generalizing the well-known extended Kalman filter, cf.~\Cref{rem:HOEKF}.
    \item An implementation of the higher order Kalman filter for two academic examples illustrates the feasibility and effectiveness of the approximation. The application to a low-dimensional chaotic oscillator indicates convergence to the Mortensen observer with increasing orders of the HOEKF. In addition, a realization for a semi-discretized, cubic wave equation demonstrates the effectiveness for systems of larger dimensions. The results are compared to available data-driven methods for the realization of the Mortensen observer, showing that the HOEKF outperforms them in accuracy and computational efficiency.  
\end{itemize}
%
%%%%%%%%%%%%%%%%%%%%%%%%%%%%%%%%%%%%%%%%%%%%%%%%%%%%%%%%%%%%%%%%%%%%%%%%%%%%%%%
%%%%%%%%%%%%%%%%%%%%%%%%%%%%%%%%%%%%%%%%%%%%%%%%%%%%%%%%%%%%%%%%%%%%%%%%%%%%%%%
\section{Some results from tensor calculus}
%%%%%%%%%%%%%%%%%%%%%%%%%%%%%%%%%%%%%%%%%%%%%%%%%%%%%%%%%%%%%%%%%%%%%%%%%%%%%%%
%%%%%%%%%%%%%%%%%%%%%%%%%%%%%%%%%%%%%%%%%%%%%%%%%%%%%%%%%%%%%%%%%%%%%%%%%%%%%%%
%
For deriving a sequence $\TenDeriv{k} \mathcal{V}$ locally approximating the Mortensen observer and characterizing the required terms via a coupled set of differential equations, we identify higher order derivatives with multidimensional arrays or tensors. Here, we restrict ourselves to the finite-dimensional case $\R^{n_1,\dots,n_d}$ and refer to \cite{Hac19} for an in-depth treatment of tensor products for general Banach spaces. Most of the presented material is well-known in the context of tensor calculus and our exposition largely follows the ones from \cite{Wol19,Sal22}.

\subsection{Basic notions and preliminaries}
\label{subsec:tensor_prelim}

We denote by $\mathbf{T}$ a \emph{tensor} of order $d$, i.e., a mapping 
\begin{equation}\label{EQ:Tensor}
    \mathbf{T} \colon \N_{n_1}\times \dots \times \N_{n_d} \to \R,\quad (i_1,\dots,i_d)\mapsto \mathbf{T}[i_1,\dots,i_d].
  \end{equation}
  The set of all such tensors is denoted by  
  \[
  \R^{n_1,\dots, n_d}\coloneqq \left\{\mathbf{T}\mid \mathbf{T}\colon\N_{n_1}\times \dots \times \N_{n_d} \to \R \right\}.
  \]
We refer to $n_\mu$ as the \emph{$\mu$-th dimension} of $\mathbf{T}$ and to $d\in\N$ as the \emph{order} of $\mathbf{T}$. 
In particular, we identify matrices as order-two tensors via $\mathbf{C}[i,j] = C_{ij}$ for $C\in\R^{n,m}$.
Given $0<\ell<d$ and two tensors $\mathbf{A}\in \R^{n_1,\dots,n_\ell},\mathbf{B}\in \R^{n_{\ell+1},\dots,n_d}$ of order $\ell$ and $d-\ell$, respectively, the \emph{tensor product} is defined by
\[
    \otimes\colon  \R^{n_1,\dots,n_\ell}\times\R^{n_{\ell+1},\dots,n_d}\to \R^{n_1,\dots,n_d}, \quad \left(\mathbf{A}\otimes \mathbf{B}\right)[i_1,\dots,i_d] \coloneqq \mathbf{A}[i_1,\dots,i_\ell]\mathbf{B}[i_{\ell+1},\dots i_d].
  \]
Similarly, for the tensor product of two tensor spaces, we set
\[
    \R^{n_1,\dots,n_\ell}\otimes\R^{n_{\ell+1},\dots,n_d} \coloneqq \vspan \left\{ \left(\mathbf{A}\otimes \mathbf{B}\right)\mid \mathbf{A}\in \R^{n_1,\dots,n_\ell}, \mathbf{B}\in \R^{n_{\ell+1},\dots,n_d}  \right\}.
  \]
When all the $n_i$ are equal, i.e., $n_1,\dots,n_d=n$, we use the shorthand notation $\bigotimes^d\R^n.$
Below, we consider successive differentiations of the HJB equation \eqref{eq: HJB} for which we utilize a generalized product rule for tensors. For this purpose, it is convenient to make use of appropriate \emph{reshapings} of the underlying multidimensional arrays. Given $d\in\N$ and a permutation $\sigma\in S_d$, a \emph{reshape} is an isomorphism defined as
  \[
    \reshape_\sigma\colon \R^{n_1,\dots,n_d} \to \R^{n_{\sigma(1)},\dots,n_{\sigma(d)}}, \quad \reshape_\sigma(\mathbf{A})[i_1,\dots,i_d] = \mathbf{A}[i_{\sigma(1)},\dots,i_{\sigma(d)}].
  \]
In the following we shorten the above notation and often simply write $\mathbf{A}^\sigma$ instead of $\reshape_\sigma(\mathbf{A})$.

\begin{remark}
  Let us note that there are different ways of defining the reshape operator. In particular, the one used in this manuscript differs from the one used in \cite{Wol19,Sal22} which reads 
  \[
    \widetilde{\reshape}_\sigma \mathbf{A}[i_{\sigma(1)},\dots,i_{\sigma(d)}] = \mathbf{A}[i_1,\dots,i_d].
  \]
  In other words, the difference in notation is caused by the implicit use of the inverse permutation and we have the following relation
  \[
    \reshape_\sigma(\mathbf{A}) = \widetilde{\reshape}_{\sigma^{-1}} \mathbf{A}.
  \]
  While the notation introduced in \cite{Wol19,Sal22} corresponds to the way the reshape operation is encoded in \Matlab, our notation allows for a more intuitive way of deriviving the differentiation rules in \Cref{S:DifferentiationRules} and has been used in a similar context in \cite{BreKP19}.
\end{remark}
Besides the tensor product, we also need a generalization of the matrix multiplication, typically referred to as \emph{tensor contraction.}
Let $\mathbf{A}\in \R^{n_1,\dots,n_d}$ and $\mathbf{B}\in \R^{m_1,\dots,m_\ell}$ be two tensors of order $d$ and $\ell$ respectively. The contraction $\mathbf{A}\tenp{\mu}{\nu} \mathbf{B}$ of the $\mu$-th dimension of $\mathbf{A}$ with the $\nu$-th  dimension of $\mathbf{B}$ is defined entrywise as
  \begin{multline*}
    \left(\mathbf{A}\tenp{\mu}{\nu} \mathbf{B}\right)[i_1,\dots,i_{\mu-1},i_{\mu+1},\dots,i_d,j_1,\dots,j_{\nu-1},j_{\nu+1},\dots,j_\ell] \\
    \coloneqq \sum_{k=1}^{n_\mu} \mathbf{A}[i_1,\dots,i_{\mu-1},k,i_{\mu+1},\dots,i_d] \mathbf{B}[j_1,\dots,j_{\nu-1},k,j_{\nu+1},\dots, j_\ell].
  \end{multline*}
For the contraction of two tensors to be well-defined,  it is required that $n_\mu=m_\nu$.
The resulting tensor $\mathbf{A}\tenp{\mu}{\nu} \mathbf{B}\in \R^{n_1,\dots,n_{\mu-1},n_{\mu+1},\dots, n_d,m_1,\dots,m_{\nu-1},n_{\nu+1},m_\ell}$ is of order $d+\ell-2$.
\begin{remark}
    For fixed dimensions, the contraction operation is associative. However, one downside of the introduced notation is that whenever the order of operations is switched, the dimensions in the contractions need to be adjusted. 
As an example consider $\mathbf{A}\in\R^{n_1,n_2,n_3}$, $\mathbf{B}\in\R^{m_2,n_2,p}$ and $\mathbf{C}\in\R^{m_1,m_2,m_3}$. Then it holds that
\begin{align*}
  \mathbf{A}\tenp{2}{1}(\mathbf{B}\tenp{1}{2}\mathbf{C})
  & = \sum_{j=1}^{n_2} \sum_{k=1}^{m_2} \mathbf{A}[i_1,j,i_2]\mathbf{B}[k,j,i_3]\mathbf{C}[i_4,k,i_5]
  = \left(\mathbf{A}\tenp{2}{2}\mathbf{B}\right)\tenp{3}{2}\mathbf{C}.
\end{align*}
Note that the term  $(\mathbf{A}*_{2,1}\mathbf{B})*_{1,2}\mathbf{C}$ is not even well-defined since in general $n_1\neq m_2$. Thus, whenever we switch the order of operation we need to also adjust the syntax.
\end{remark}

When deriving the observer equations, we often need to contract a tensor and a matrix followed by applying a certain permutation. Thus, for $\mathbf{A}\in \R^{n_1,\dots,n_d}$ and $\mathbf{C}\in \R^{n_\mu,m}$ we define 
\begin{align*}
  \left(\mathbf{A} \times_\mu \mathbf{C}\right)[i_1,\dots,i_k] \coloneqq& \sum_{k=1}^{n_\mu} \mathbf{A}[i_1,\dots,i_{\mu-1},k,i_{\mu+1},\dots,i_d]\mathbf{C}[k,i_\mu]\\
  =& \left(  \mathbf{A} \tenp{\mu}{1} \mathbf{C} \right)^{(1,\dots,\mu-1,\mu+1,\dots,d,\mu)}.
\end{align*}
Here, we contract the $\mu$-th dimension of $\mathbf{A}$ with the first dimension of $\mathbf{C}$ and then perform a reshape operation to restore the order of the dimensions of $\mathbf{A}$ and substitute the contracted dimension with the remaining dimension of $\mathbf{C}$.

\subsection{Differentiation rules for tensors}\label{S:DifferentiationRules}

In this subsection, we recall some basics on the derivatives of vector-valued functions and identify these derivatives with tensors. We also state a generalized product rule that already has been utilized in \cite{BreKP19}.

First, recall that the $k$-th derivatives of a multivariable function $u\colon\R^n\to\R$ can be identified with a $k$-linear map $D^{k} u\colon (\R^n)^{\times k}\to \R$.

Due to the natural connection between tensors and multilinear mappings, given $u\in C^k(\R^n;\R)$ and $g\in C^k(\R^n;\R^m)$, we define
\begin{align*}
  \TenDeriv{k} u(\xi)[j_1,\dots,j_k] \coloneqq \partial_{j_k,\dots,j_1} u(\xi)
\quad \text{and} \quad
  \TenDeriv{k}g(\xi)[i,j_1,\dots,j_k] \coloneqq \partial_{j_k,\dots,j_1} g_i(\xi).
\end{align*}
This way we have that 
\[
  D^ku(\xi)(e_{j_1},\dots,e_{j_k}) = \partial_{j_k,\dots,j_1}u(\xi) = \TenDeriv{k}u(\xi)[j_1,\dots,j_k],
\]
where $D^ku(v_1,\dots,v_k)$ is the $k$-th directional derivative of $u$,  first with respect to $v_1$ and last with respect to $v_k$ and $e_j$ denotes the $j$-th unit vector.
\begin{remark}\label{REM:TensorDeriv}
  ~
  \begin{enumerate}[topsep=0pt]
      \item Note that for fixed $i$, we have that
      \[
      \TenDeriv{k}g(\xi)[i,j_1,\dots,j_k]=\TenDeriv{k}g_i(\xi)[j_1,\dots,j_k].
      \]
      \item Similar to \cref{eq: Hessian_V}, we can identify $\TenDeriv{1}u$, $\TenDeriv{2}u$ and $\TenDeriv{1}g$ with the gradient $\nabla u$, Hessian $\nabla^2 u$ and Jacobian $\mathcal{J}_g$, respectively. Technically,  $\TenDeriv{2}u$  corresponds to the transpose of the Hessian, due to the ordering of its entries. However, since we assumed continuous differentiability, by the theorem of Schwarz the Hessian is symmetric and hence the identification with $\TenDeriv{2}u$ is justified.
  \end{enumerate}
\end{remark}
Next, we recall a generalized product rule that can be found in, e.g., \cite[Proposition 5]{Har06} or \cite[Lemma 12]{BreKP19}. For this purpose, let us define the set of permutation $S_{i,j}\subset S_m$ as follows.
\begin{equation*}
  S_{i,j} \coloneqq \left\{ \sigma \in S_{i+j} \mid \sigma(1)<\dots<\sigma(i) \text{ and } \sigma(i+1)<\dots< \sigma(i+j) \right\}.
\end{equation*}
Note that $|S_{i,j}|= \binom{i+j}{i}$.
We then have the following result.
\begin{lemma}\label{LEM:PartDivProductRule}
  Let $u,v\in C^k(\R^n;\R)$. Then it holds that
  \begin{equation*}
    \partial_{j_1,\dots,j_k} (uv) = \sum_{i=0}^{k} \sum_{\sigma \in S_{i,k-i}} (\partial_{j_{\sigma(1)}, \dots ,j_{\sigma(i)}} u) (\partial_{j_{\sigma(i+1)}, \dots ,j_{\sigma(k)}}v).
  \end{equation*}
\end{lemma}

The proof can be found in \cite{BreKP19} in the language of multilinear forms.
Introducing the operator
\[
  \Sym_{i,j}(\mathbf{T}) = \sum_{\sigma \in S_{i,j}} \mathbf{T}^{\sigma}
\]
for $\mathbf{T}\in \bigotimes^{i+j}\R^n$,
we can write the derivative of $uv$ as
\begin{equation}\label{EQ:ScalarProductRule}
  \TenDeriv{k}(uv) = \sum_{i=0}^{k} \Sym_{i,k-i} \left( \TenDeriv{i}u \otimes \TenDeriv{k-i}v \right).
\end{equation}
\begin{remark}
  Note that in \cite{BreKP19} the operator $\Sym_{i,j}(\mathbf{T})$ is further multiplied by $\binom{i+j}{i}^{-1}$ to represent the symmetric part of the tensor $\mathbf{T}$. We choose to omit the scaling factor in the definition of $\Sym_{i,j}$, since it cancels out in the equations \eqref{EQ:ScalarProductRule} and \eqref{EQ:ProductRule}.
\end{remark}

Later we need to differentiate the contraction of two vector-valued functions. The previous product rule can easily be generalized to this setting.
\begin{lemma}\label{LEM:ProductRule}
Let $u,g\in C^k(\R^n;\R^m)$, then it holds that
  \begin{equation}\label{EQ:ProductRule}
    \TenDeriv{k}(u^\top g) = \sum_{i=0}^{k} \Sym_{i,k-i}\left( \TenDeriv{i}u \tenp{1}{1} \TenDeriv{k-i}g \right).
  \end{equation}
\end{lemma}
\begin{proof}
  Note that differentiation as well as the operator $\Sym_{i,j}$ are linear. By the first part of \Cref{REM:TensorDeriv}, it holds that
  \[
    \TenDeriv{k}g(\xi)[i,j_1,\dots,j_k]=\TenDeriv{k}g_i(\xi)[j_1,\dots,j_k].
  \]
  Thus we get that
  \begin{align*}
    \TenDeriv{k}(u^\top g) =&\; \TenDeriv{k}\left(\sum_{j=0}^{m} g_ju_j \right) = \sum_{j=0}^{m} \TenDeriv{k}\left( g_ju_j \right).
  \end{align*}
  Now using \cref{EQ:ScalarProductRule} we obtain
  \begin{align*}
    \TenDeriv{k}(u^\top g) &= \sum_{j=0}^{m} \sum_{i=0}^{k} \Sym_{i,k-i} \left( \TenDeriv{i}u_j \otimes \TenDeriv{k-i}g_j \right)= \sum_{i=0}^{k} \Sym_{i,k-i}\left( \TenDeriv{i}u \tenp{1}{1} \TenDeriv{k-i}g \right). 
  \end{align*}
\end{proof}
%
%%%%%%%%%%%%%%%%%%%%%%%%%%%%%%%%%%%%%%%%%%%%%%%%%%%%%%%%%%%%%%%%%%%%%%%%%%%%%%%
%%%%%%%%%%%%%%%%%%%%%%%%%%%%%%%%%%%%%%%%%%%%%%%%%%%%%%%%%%%%%%%%%%%%%%%%%%%%%%%
\section{Higher order extended Kalman filters}
%%%%%%%%%%%%%%%%%%%%%%%%%%%%%%%%%%%%%%%%%%%%%%%%%%%%%%%%%%%%%%%%%%%%%%%%%%%%%%%
%%%%%%%%%%%%%%%%%%%%%%%%%%%%%%%%%%%%%%%%%%%%%%%%%%%%%%%%%%%%%%%%%%%%%%%%%%%%%%%
%
Using the results from \Cref{LEM:ProductRule}, in this section, we derive a sequence of approximations to the Mortensen observer by successively taking derivatives of the HJB equation \eqref{eq: HJB} and assuming the derivative of a certain order to vanish. Let us demonstrate the main idea by exemplarily considering a second order approximation of the Mortensen observer. For this purpose, assume that we construct an approximation $\widehat{x}(t)\approx \widehat{x}_{\mathrm{M}}(t)$ by making the ansatz $\TenDeriv{3} \mathcal{V}(\cdot,\widehat{x}(\cdot)) = 0$. Denoting $P(t) = \nabla^2 \mathcal{V}(t,\widehat{x}(t))$, equations \eqref{eq: observerEq} and \eqref{eq: Hessian_V} completely describe the resulting observer as

\begin{equation}\label{EQ:SecondOrderObserver}
 \begin{aligned}
    \dot{\widehat{x}}(t) &= f(\widehat{x}(t)) 
	+ P(t)^{-1} \mathcal{J}_h(\widehat{x}(t))^\top Q (y(t) - h(\widehat{x}(t) ) ),
    \\
    \dot{P}(t) &=
    -\mathcal{J}_f(\widehat{x}(t))^\top  P(t)
    - P(t) \mathcal{J}_f(\widehat{x}(t)) 
    -  P(t) G R^{-1} G^\top  P(t)  \\
    &\quad + \mathcal{J}_h(\widehat{x}(t))^\top Q \mathcal{J}_h(\widehat{x}(t)) 
    - \TenDeriv{2} h (\widehat{x}(t))  \tenp{1}{1} Q (y(t)-h(\widehat{x}(t))), 
    \\
	\widehat{x}(0) &= x_0 \quad\text{and}\quad   P(0) = \Gamma.
  \end{aligned}
\end{equation}
As we can see from \eqref{eq: ext Kalman} and its discussion, the observer $\widehat{x}$ is in spirit similar to the extended Kalman filter and only differs in the term corresponding to the second derivative of $h$. Note that in order to define this observer we need to assume $f$ and $h$ to be $C^2$-regular and $\mathcal{V}$ to be $C^3$-regular. 

Analogously, as we have done for the second order approximation, let us now approximate the Mortensen observer by considering all derivatives of the value function of order up to some fixed $k\ge 3$, which is fixed for the remainder of the section. Hence, we neglect all terms containing $\TenDeriv{\ell} \mathcal{V}(t,\xi)$ for $\ell=k$. We will see that the derivatives of order larger than two are characterized by generalized differential Lyapunov equations. Differentiating the HJB equation requires the following assumptions.
\begin{assumption}~ \label{ass:smoothness}
    \begin{enumerate}
        \item The functions $f \colon \mathbb{R}^n \to \mathbb{R}^n$ and $h \colon \mathbb{R}^n \to \mathbb{R}^p$ are $k$-times continuously differentiable.
        \item For all $t \in [0,T]$ the value function $\mathcal{V}(t,\cdot)$ is $(k+1)$-times continuously differentiable.
        \item The value function and its spatial derivatives are $C^1$ in time and the order of derivatives can be exchanged, i.e., for $\ell \leq k $ it holds
        $ \partial_t \TenDeriv{\ell} \mathcal{V}(\cdot,\widehat{x}_\mathrm{M}(\cdot)) = \TenDeriv{\ell} \partial_t \mathcal{V}(\cdot,\widehat{x}_\mathrm{M}(\cdot)) $.
    \end{enumerate}
\end{assumption}
Two additional assumptions are required to ensure that the Mortensen observer is well-defined and satisfies \cref{eq: observerEq}.
\begin{assumption}~ \label{ass:mortensen}
    \begin{enumerate}
        \item For every $t \in [0,T]$ the Mortensen observer $\widehat{x}_\mathrm{M}(t)$ is well-defined as the unique minimizer of the value function $\mathcal{V}(t,\cdot)$.
        \item For all $t \in [0,T]$ the Hessian $\nabla^2 \mathcal{V}(t,\widehat{x}_\mathrm{M}(t))$ is a regular matrix.
    \end{enumerate}
\end{assumption}
Note that, while the assumptions on the regularity of the value function may seem restrictive, we point out that for quadratic dynamics with linear measurements the value function $\mathcal{V}$ has been shown to be locally smooth with a positive definite Hessian \cite{BreS24}.
Under these assumption we state the main theoretical result of this work.
\begin{theorem}\label{thm:main}
    Let \Cref{ass:smoothness} and \Cref{ass:mortensen} hold and additionally assume $\TenDeriv{k+1} \mathcal{V}(\cdot,\widehat{x}_\mathrm{M}(\cdot)) = 0$.
    Then $\widehat{x}_\mathrm{M}$ is characterized as the solution $\widehat{x}$ of 
    \begin{equation}\label{EQ:HigherOrderObserver}
    \begin{aligned}
        \dot{\widehat{x}}(t) 
        =& f(\widehat{x}(t)) + \mathbf{P}_2(t)^{-1} (\mathcal{J}_h(\widehat{x}(t)))^\top  (y(t) - h(\widehat{x}(t))), 
        \\
        \dot{\mathbf{P}}_2(t) 
        =& - \J_f(\widehat{x}(t))^\top  \mathbf{P}_2(t) - \mathbf{P}_2(t) \J_f(\widehat{x}(t)) - \mathbf{P}_2(t) GR^{-1}G^\top  \mathbf{P}_2(t)   - \mathbf{R}_2(t) 
        \\
        &\qquad - \mathbf{P}_3(t) \tenp{1}{1} \left(\mathbf{P}_2(t)^{-1} \left(\mathcal{J}_h(\widehat{x}(t))^\top Q\left(y(t) - h(\widehat{x}(t))\right)\right)\right), 
        \\
        \dot{\mathbf{P}}_j(t) 
        =& -\sum_{i=1}^{j} \mathbf{P}_j(t) \times_i \left(\J_f(\widehat{x}(t)) + G R^{-1} G^\top  \mathbf{P}_2(t) \right) - \mathbf{R}_j(t) 
        \\
        &\qquad - \mathbf{P}_{j+1}(t) \tenp{1}{1}\left(\mathbf{P}_2(t)^{-1} \left(\mathcal{J}_h(\widehat{x}(t))^\top Q\left(y(t) - h(\widehat{x}(t))\right)\right)\right),
        \\
        \dot{\mathbf{P}}_k(t)  
        =& -\sum_{i=1}^{k} \mathbf{P}_k(t) \times_i \left( \J_f(\widehat{x}(t)) + G R^{-1}G^\top  \mathbf{P}_2(t) \right) - \mathbf{R}_k(t),
    \end{aligned}
    \end{equation}
with the initial conditions
\[
\widehat{x}(0) = x_0, \qquad \mathbf{P}_2(0) = \Gamma
\quad \text{and} \quad \mathbf{P}_j(0) = 0 \quad\text{for all } j = 3,\dots,k
\]
where 
$\mathbf{R}_\ell$ is given by
\begin{equation}\label{eq:HOEKF_R}
    \begin{aligned}
        \mathbf{R}_\ell(t) &=
        \sum_{i=1}^{\ell-2} \Sym_{i,\ell-i}\left( \mathbf{P}_{i+1}(t)  \tenp{1}{1}\TenDeriv{\ell-i}f(\widehat{x}(t)) \right)
        \\
        &\qquad + \frac{1}{2}\sum_{i=2}^{\ell-2} \Sym_{i,\ell-i}\left( \mathbf{P}_{i+1} (t) \tenp{1}{1} \left( G R^{-1} G^\top \tenp{2}{1}\mathbf{P}_{\ell-i+1} (t) \right) \right)
        \\
        &\qquad + \TenDeriv{\ell} h(\widehat{x}(t)) \tenp{1}{1} Q  (y(t) - \widehat{x}(t))
        \\
        & \qquad - \frac{1}{2} \sum_{i=1}^{\ell -1} \Sym_{i,\ell-i} \left(\TenDeriv{i} h(\widehat{x}(t)) \tenp{1}{1} \left( Q \tenp{2}{1} \TenDeriv{\ell-i} h(\widehat{x}(t))\right) \right)
    \end{aligned}
\end{equation}
for $\ell= 2,\dots,k$.
\end{theorem}
Before proving the theorem some remarks are in order.
\begin{remark}\label{rem:HOEKF}
    The assumption of a vanishing $(k+1)$ order derivative $\TenDeriv{k+1} \mathcal{V}(\cdot,\widehat{x}_\mathrm{M}(\cdot))$ can generally not be expected to hold. Hence, in a general setting, the trajectory $\widehat{x}$ obtained from \eqref{EQ:HigherOrderObserver} is to be understood as an approximation rather than an analytic representation of the Mortensen observer $\widehat{x}_\mathrm{M}$.
    We call the resulting observer the \textbf{higher order extended Kalman filter of order $k$} or short \emph{HOEKF}-$k$.
\end{remark}
Motivated by the generalized Lyapunov equations in \cite{BreKP19} we refer to the equations for $\mathbf{P}_j$ and $\mathbf{P}_k$ as \emph{generalized differential Lyapunov equations}.
The higher order extended Kalman filter defined by \eqref{EQ:HigherOrderObserver}-\eqref{eq:HOEKF_R} consists of a feedback equation, a Riccati-like equation and $k-2$ generalized differential Lyapunov equations, all of which are coupled. 
Note that, when applied to linear systems, i.e., $f(\xi) = A \xi$ and $h(\xi) = C \xi$, for all $k \geq 2$ the HOEKF-$k$ reduces to the Kalman filter
\begin{equation}\label{eq:KF}
\begin{alignedat}{2}
    \dot{\widehat{x}}_\mathrm{K}(t) &= A \widehat{x}_\mathrm{K}(t) + \Sigma(t) C^\top Q (y(t) - C\widehat{x}_\mathrm{K}(t)),
    \qquad \qquad
    &&\widehat{x}_\mathrm{K}(0) = x_0,\\
    \dot{\Sigma}(t) &= A \Sigma(t) + \Sigma(t) A^\top - \Sigma(t) C^\top Q C \Sigma(t) + G R^{-1} G^\top,
    \qquad
    &&\Sigma(0) = \Gamma^{-1}.
\end{alignedat}
\end{equation}
This can be seen by observing that for linear systems it holds $\mathbf{R}_3(t) = 0$ and  
\begin{equation*}
    \mathbf{R}_j(t) 
    =
    \frac{1}{2} \sum_{i=2}^{j-2} \Sym_{i,j-i} \left( \mathbf{P}_{i+1}(t) \tenp{1}{1} \left( GG^\top \tenp{2}{1} \mathbf{P}_{j-i+1}(t) \right) \right),
    \qquad j\geq 4. 
\end{equation*}
Consequently, a solution $(\widehat{x},(\mathbf{P}_j)_{j=1}^k)$ to \eqref{EQ:HigherOrderObserver} is given by $\widehat{x} = \widehat{x}_\mathrm{K}$, $\mathbf{P}_2 = \Sigma^{-1}$, and  $\mathbf{P}_j = 0$, $j \geq 3$, with $\widehat{x}_\mathrm{K}$ and $\Sigma$ solving \eqref{eq:KF}.
\begin{remark}[Existence of solutions]~\\
At first glance, the structure of the generalized differential Lyapunov equations seems quite appealing. Indeed, it is similar to the equations appearing in \cite{BreKP19} for which the authors show existence of solutions. In the present setting, the coupling and the matrix inverse introduce additional challenges prohibiting an analogous treatment. The analytical investigation regarding well-posedness of the observer hence requires a novel approach and is left for future work.

\end{remark}

\begin{remark}
  There are numerous ways to modify the previous considerations.
      E.g., in \cite{Fle97}, the author considers the cost functional
      \[
        \tilde{J}(t,\xi,v) = \phi(x(0)) + \frac{1}{2}\int_{0}^{t}\norm{v(s)}^2 + \alpha\norm{\tilde{y}(s)-Cx(s)}^2~\mathrm{d} s.
      \]
      Using this cost functional, we end up with different initial conditions for the equations in \eqref{EQ:HigherOrderObserver}. In particular, the corresponding initial condition to the HJB equation \eqref{eq: HJB} would take the form $\mathcal{V}(0,\xi) = \phi(\xi)$. Hence, by assuming sufficient smoothness of $\phi$, the initial conditions of the $\mathbf{P}_i$ would take the form $\mathbf{P}_i(0) = \TenDeriv{i}\phi(x_0)$ for $i\ge 2$.
\end{remark}

The remainder of this section is concerned with the proof of \Cref{thm:main}.
As a first step we differentiate both sides of the HJB equation \eqref{eq: HJB} with respect to $\xi$.
\begin{lemma}\label{LEM:HigherDerivativeHJB}
    Let \Cref{ass:smoothness} hold. Then for all $3 \leq j \leq k$ it holds
    \begin{equation}
    \begin{aligned}\label{EQ:HigherDerivativeHJB}
        \partial_t \TenDeriv{j} \mathcal{V}(t,\xi) &= -\sum_{i=0}^{j} \Sym_{i,j-i}\left( \TenDeriv{i+1} \mathcal{V}(t,\xi) \tenp{1}{1}\TenDeriv{j-i}f(\xi) \right)
        \\ 
        &\qquad - \frac{1}{2}\sum_{i=0}^{j} \Sym_{i,j-i}\left( \TenDeriv{i+1} \mathcal{V}(t,\xi) \tenp{1}{1} \left( G R^{-1}G^\top \tenp{2}{1}\TenDeriv{j-i+1} \mathcal{V}(t,\xi) \right) \right)
        \\
        &\qquad - \TenDeriv{j} h(\xi) \tenp{1}{1} Q y(t) 
        \\
        & \qquad + \frac{1}{2} \sum_{i=0}^{j} \Sym_{i,j-i} \left(\TenDeriv{i} h(\xi) \tenp{1}{1} \left(Q \tenp{2}{1} \TenDeriv{j-i} h(\xi)\right) \right).
    \end{aligned}
    \end{equation}
\end{lemma}
\begin{proof}
  We differentiate the three summands of \eqref{eq: HJB} separately using \Cref{LEM:ProductRule}. 
  For the first term, we obtain
  \begin{align*}
    \TenDeriv{j} \left(  \nabla \mathcal{V}(t,\xi)^\top f(\xi)\right) =& \sum_{i=0}^{j} \Sym_{i,j-i} \left( \TenDeriv{i} \left( \nabla \mathcal{V}(t,\xi) \right) \tenp{1}{1} \TenDeriv{j-i}f(\xi) \right) \\
    =&  \sum_{i=0}^{j} \Sym_{i,j-i} \left( \TenDeriv{i+1} \mathcal{V}(t,\xi) \tenp{1}{1} \TenDeriv{j-i}f(\xi) \right).
  \end{align*}
  Similarly, for the second term, we see that
  \begin{align*}
    \TenDeriv{j}\bigl(\Vert G^\top \nabla \mathcal{V}(t,\xi) \Vert_{R^{-1}}^2 \bigr) &=
    \TenDeriv{j}\bigl( \nabla \mathcal{V}(t,\xi)^\top  (  G R^{-1}G^\top  \nabla \mathcal{V}(t,\xi) ) \bigr) \\
    &= \sum_{i=0}^{j} \Sym_{i,j-i} \left( \TenDeriv{i} \left( \nabla \mathcal{V}(t,\xi) \right) \tenp{1}{1} \TenDeriv{j-i}\left( G R^{-1}G^\top  \nabla \mathcal{V}(t,\xi) \right)  \right)\\
    &= \sum_{i=0}^{j}\Sym_{i,j-i} \left( \TenDeriv{i+1}  \mathcal{V}(t,\xi) \tenp{1}{1} \left(G R^{-1}G^\top \tenp{2}{1}\TenDeriv{j-i+1} \mathcal{V}(t,\xi) \right)  \right).
  \end{align*}
  Finally differentiating the third term, we obtain
  \begin{align*}
      \TenDeriv{j}\left(\Vert y(t) - h(\xi) \Vert_{Q}^2\right) 
      &= \TenDeriv{j}\left( y(t)^T Q y(t) -2  h(\xi)^\top Q y(t) + h(\xi)^\top Q h(\xi) \right)
      \\
      &= \TenDeriv{j}\left( y(t)^T Q y(t)\right) -2  \TenDeriv{j}\left(h(\xi)^\top Q y(t)\right) + \TenDeriv{j}\left(h(\xi)^\top Q h(\xi) \right)
      \\ 
      &= -2   \TenDeriv{j} h(\xi) \tenp{1}{1} Q y(t) 
      \\
      & \qquad + \sum_{i=0}^{j} \Sym_{i,j-i} \left(\TenDeriv{i} h(\xi) \tenp{1}{1} \left(Q \tenp{2}{1} \TenDeriv{k-i} h(\xi)\right) \right).
  \end{align*}
\end{proof}

We further require the following two technical lemmas. 
\begin{lemma}\label{LEM:HigherDerivativesCalculations}
  Let $\mathbf{A}\in \bigotimes^{j+1}\R^n$ and $\mathbf{B}\in \bigotimes^{\ell+1}\R^n$ be symmetric, and let $\Gamma=\Gamma^\top \in\R^{n,n}$ be a symmetric matrix. Then it holds that
      \begin{equation}\label{EQ:SymmetrieHJBValueFuncTerms}
        \Sym_{j,\ell}\left(\mathbf{A} \tenp{1}{1}\left(\Gamma\tenp{2}{1} \mathbf{B}\right)\right) = \Sym_{\ell,j}\left(\mathbf{B} \tenp{1}{1}\left(\Gamma\tenp{2}{1} \mathbf{A}\right)\right).
      \end{equation}
\end{lemma}
\begin{proof}
  We rewrite the first expression as
  \begin{align*}
    &\Sym_{j,\ell}\left(\mathbf{A} \tenp{1}{1}\left(\Gamma\tenp{2}{1} \mathbf{B}\right)\right)[i_1,\dots,i_{j+\ell}]\\ =& \sum_{\sigma\in S_{j,\ell}} \sum_{s=1}^{n} \mathbf{A}[s,i_{\sigma(1)},\dots,i_{\sigma(j)}] \sum_{r=1}^{n} \Gamma_{sr}\mathbf{B}[r,i_{\sigma(j+1)},\dots,i_{\sigma(j+\ell)}]\\
    =& \sum_{\sigma\in S_{j,\ell}} \sum_{r=1}^{n} \mathbf{B}[r,i_{\sigma(j+1)},\dots,i_{\sigma(j+\ell)}]  \sum_{s=1}^{n}  \Gamma_{sr} \mathbf{A}[s,i_{\sigma(1)},\dots,i_{\sigma(j)}] \\
    =& \sum_{\sigma\in S_{j,\ell}} \left( \mathbf{B} \tenp{1}{1} \left( \Gamma\tenp{2}{1} \mathbf{A} \right) \right)[i_{\sigma(j+1)},\dots,i_{\sigma(j+\ell)},i_{\sigma(1)},\dots,i_{\sigma(j)}].
  \end{align*}
  Defining the permutation $\tau$ as
  \[
    (\tau(1),\dots,\tau(j+\ell)) = (j+1,\dots,j+\ell,1,\dots,j),
  \]
  it remains to show that
  \[
    S_{\ell,j} = \left\{ \sigma\circ\tau \mid \sigma\in S_{j,\ell} \right\}.
  \]
  To that end let $\tilde{\sigma}=\sigma\circ\tau \in \left\{ \sigma\circ\tau \mid \sigma\in S_{j,\ell} \right\}$. It then follows
  \begin{align*}
    \tilde{\sigma}(1) &= \sigma(\tau(1)) = \sigma(j+1) < \dots < \sigma(j+\ell) =\sigma(\tau(\ell)) = \tilde{\sigma}(\ell) \\
    \tilde{\sigma}(\ell+1) &= \sigma(\tau(\ell+1)) = \sigma(1) < \dots < \sigma(j) =\sigma(\tau(j+\ell)) = \tilde{\sigma}(j+\ell)
  \end{align*}
  and therefore $\tilde{\sigma}\in S_{\ell,j}$. Note that also $|S_{j,\ell}|=|S_{\ell,j}|<\infty$ and hence, equality between the two sets of permutations follows.
\end{proof}

\begin{lemma}\label{LEM:HigherDerivativesCalculations2}
  Let $\mathbf{A}\in \bigotimes^\ell\R^n$ be symmetric and $\mathbf{B} \in \R^{n,n}$. Then it holds that
  \begin{equation*}
      \Sym_{\ell-1,1}\left( \mathbf{A} \tenp{1}{1} \mathbf{B} \right) = \sum_{j=1}^{\ell} \mathbf{A} \times_j \mathbf{B}.
  \end{equation*}
\end{lemma}
\begin{proof}
  First note that
  \[
    S_{\ell-1,1}= \left\{ (1,\dots,\ell),(1,3,\dots,\ell,2),\dots,(1,\dots,\ell-2,\ell,\ell-1) \right\}.
  \]
 By symmetry of $\mathbf{A}$ we have $\mathbf{A}*_{1,1}\mathbf{B}=\mathbf{A}*_{j,1}\mathbf{B}$ for all $j=1,\dots \ell$ and by writing out the product, it follows
\begin{align*}
    \Sym_{\ell-1,1}\left( \mathbf{A} \tenp{1}{1} \mathbf{B} \right)[i_1,\dots,i_\ell] =&  \sum_{\sigma\in S_{\ell-1,1}} \left( \mathbf{A} \tenp{1}{1} \mathbf{B} \right)[i_{\sigma(1)},\dots,i_{\sigma(\ell)}] \\
    =& \sum_{j=1}^{\ell}\left( \mathbf{A} \tenp{1}{1} \mathbf{B} \right)[i_1,\dots,i_{j-1},i_{j+1},\dots,i_\ell,i_j]\\
    =& \sum_{j=1}^{\ell} \left(\mathbf{A} \times_j \mathbf{B} \right)[i_1,\dots,i_\ell],
\end{align*}
which concludes the proof.
\end{proof}
With these results at hand we are in a position to prove the main result of this work.\\ 
\textbf{Proof of \Cref{thm:main}}
\begin{proof}
We denote $\mathbf{P}_2(t) = \nabla^2 \mathcal{V}(t,\widehat{x}_\mathrm{M}(t))$ and for $j=3,\dots,k+1$ we denote $ \mathbf{P}_j(t)\coloneqq\TenDeriv{j} \mathcal{V}(t,\widehat{x}_\mathrm{M}(t)) $. 
With \Cref{ass:smoothness} and \Cref{ass:mortensen}, a straightforward extension of the arguments presented in \cite[Sec.~2.1]{BreKu21} shows that $\widehat{x}_\mathrm{M}$ satisfies the observer equation 
\begin{equation}\label{eq:obsProof}
    \dot{\widehat{x}}_\mathrm{M}(t) 
        = f(\widehat{x}_{\mathrm{M}}(t)) + \mathbf{P}_2(t)^{-1} (\mathcal{J}_h(\widehat{x}_{\mathrm{M}}(t)))^\top  (y(t) - h(\widehat{x}_{\mathrm{M}}(t))).
\end{equation}
Taking two $\xi$ derivatives on both sides of the HJB equation \eqref{eq: HJB} and evaluating along $\xi = \widehat{x}_{\mathrm{M}}(t)$ yields the announced differential matrix Riccati-type equation characterizing $\mathbf{P}_2$. Here, some terms vanish due to the fact that $\nabla \mathcal{V}(t,\widehat{x}_{\mathrm{M}}(t))=0$. To see this, note that by definition $\widehat{x}_\mathrm{M}(t)$ minimizes $\mathcal{V}(t,\cdot)$.
Now, for $j=3,\dots,k$ the chain rule yields 
\begin{equation*}
    \dot{ \mathbf{P}}_j(t)
    =
   \frac{\mathrm{d}}{\mathrm{d}t} \left[ \TenDeriv{j} \mathcal{V}(t,\widehat{x}_{\mathrm{M}}(t)) \right]
    =
    \partial_t  \TenDeriv{j} \mathcal{V}(t,\widehat{x}_{\mathrm{M}}(t)) + \TenDeriv{j+1} \mathcal{V}(t,\widehat{x}_{\mathrm{M}}(t)) \tenp{1}{1} \dot{\widehat{x}}_{\mathrm{M}}(t).
\end{equation*}
Utilizing \Cref{LEM:HigherDerivativesCalculations}, \eqref{eq:obsProof}, and $\nabla \mathcal{V}(t,\widehat{x}_{\mathrm{M}}(t))=0$ yields 
\begin{align*}
    \dot{\mathbf{P}}_j(t) 
    &= \mathbf{P}_{j+1}(t) \tenp{1}{1} \left(\mathbf{P}_2(t)^{-1} \left(\mathcal{J}_h(\widehat{x}_{\mathrm{M}}(t))^\top Q\left(y(t) - h(\widehat{x}_{\mathrm{M}}(t))\right)\right)\right) 
    \\
    &\qquad  - \Sym_{j-1,1}\left( \mathbf{P}_{j}(t)  \tenp{1}{1} \left( \TenDeriv{1}f(\widehat{x}_{\mathrm{M}}(t))+  G R^{-1}G^\top \tenp{2}{1}\mathbf{P}_{2}(t)  \right) \right)
    \\
    &\qquad -\sum_{i=1}^{j-2} \Sym_{i,j-i}\left( \mathbf{P}_{i+1}(t)  \tenp{1}{1}\TenDeriv{j-i}f(\widehat{x}_{\mathrm{M}}(t)) \right)
    \\
    &\qquad - \frac{1}{2}\sum_{i=2}^{j-2} \Sym_{i,j-i}\left( \mathbf{P}_{i+1} (t) \tenp{1}{1} \left( G R^{-1} G^\top \tenp{2}{1}\mathbf{P}_{j-i+1} (t) \right) \right)
    \\
    &\qquad - \TenDeriv{j} h(\widehat{x}_{\mathrm{M}}(t)) \tenp{1}{1} Q ( y(t) - h(\widehat{x}_{\mathrm{M}}(t))
    \\
    & \qquad + \frac{1}{2} \sum_{i=1}^{j-1} \Sym_{i,j-i} \left(\TenDeriv{i} h(\widehat{x}_{\mathrm{M}}(t)) \tenp{1}{1} \left(Q \tenp{2}{1} \TenDeriv{j-i} h(\widehat{x}_{\mathrm{M}}(t))\right) \right).
\end{align*}
Finally, with \Cref{LEM:HigherDerivativesCalculations2} together with \eqref{eq:obsProof}, we obtain the asserted tensor equations.
The initial conditions again follow directly from the initial condition of the HJB equation \eqref{eq: HJB}.
\end{proof}

%
%%%%%%%%%%%%%%%%%%%%%%%%%%%%%%%%%%%%%%%%%%%%%%%%%%%%%%%%%%%%%%%%%%%%%%%%%%%%%%%
%%%%%%%%%%%%%%%%%%%%%%%%%%%%%%%%%%%%%%%%%%%%%%%%%%%%%%%%%%%%%%%%%%%%%%%%%%%%%%%
\section{Numerical experiments}\label{sec:Num}
%%%%%%%%%%%%%%%%%%%%%%%%%%%%%%%%%%%%%%%%%%%%%%%%%%%%%%%%%%%%%%%%%%%%%%%%%%%%%%%
%%%%%%%%%%%%%%%%%%%%%%%%%%%%%%%%%%%%%%%%%%%%%%%%%%%%%%%%%%%%%%%%%%%%%%%%%%%%%%%
%
To illustrate the analytical results and to showcase the improvements of higher order observers compared to the extended Kalman filter we implement the presented scheme for two examples, namely a low-dimensional chaotic system and a semilinear hyperbolic partial differential equation discretized in the spatial variables. We note that while these experiments justify the concept and motivate further research, the design of a sophisticated algorithm taking into account the structure of the equations \eqref{EQ:HigherOrderObserver} and scaling to systems of larger dimensions will be addressed in future work. 
The \Matlab~ code relies on the tensor toolbox \cite{BaKo25} and is available in \cite{BrRaSc26}. 
All experiments were run in \Matlab R2024b. The (higher order) extended Kalman filters are realized using the \Matlab~ routine \texttt{ode45} with relative tolerance $10^{-6}$ and absolute tolerance $10^{-8}$.
The results are compared to approximations of the Mortensen observer subsequently denoted as $\widehat{x}_\mathrm{M}$. For the low-dimensional example this approximation is obtained via an integration of the observer equation \eqref{eq: observerEq} using a BDF4 scheme. Evaluations of the Hessian $\nabla^2 \mathcal{V}$ are realized by treating the associated open-loop problem using gradient descent and subsequently solving a matrix Riccati equation, cf.~\cite[Subs.~3.1, Sec.~5]{BreKuSc23}, for the details.
This approach is not feasible for the system of dimension twelve resulting from the spatial discretization of the PDE. Instead, an approximation of the Mortensen observer is obtained using the scheme proposed in 
\cite{BreKuSc23}, see \Cref{subsubsec:WaveMort} for the details, and \cite{SchB24a} for the \Matlab~ code.

%
%%%%%%%%%%%%%%%%%%%%%%%%%%%%%%%%%%%%%%%%%%%%%%%%%%%%%%%%%%%%%%%%%%%%%%%%%%%%%%%
%%%%%%%%%%%%%%%%%%%%%%%%%%%%%%%%%%%%%%%%%%%%%%%%%%%%%%%%%%%%%%%%%%%%%%%%%%%%%%%
\subsection{Duffing oscillator}
%%%%%%%%%%%%%%%%%%%%%%%%%%%%%%%%%%%%%%%%%%%%%%%%%%%%%%%%%%%%%%%%%%%%%%%%%%%%%%%
%%%%%%%%%%%%%%%%%%%%%%%%%%%%%%%%%%%%%%%%%%%%%%%%%%%%%%%%%%%%%%%%%%%%%%%%%%%%%%%
%
The system under consideration in the first example is a Duffing oscillator, which in state-space form reads
\begin{equation}\label{eq: Duffing}
	\begin{aligned}
		\frac{\mathrm{d}}{\mathrm{d}t} \begin{bmatrix} x_1(t) \\ x_2(t) \end{bmatrix} &= 
		A
		\begin{bmatrix} x_1(t) \\ x_2(t) \end{bmatrix} 
		+ 
		\begin{bmatrix}
			0 \\ 
			- \beta \, x_1(t)^3
		\end{bmatrix} + G \, v(t),~~~
		\begin{bmatrix} x_1(0) \\ x_2(0) \end{bmatrix} = x_0 + \eta,\\
		y(t) &= C \begin{bmatrix} x_1(t) \\ x_2(t) \end{bmatrix} + \mu(t),
	\end{aligned}
\end{equation}
where
\begin{equation*}
	A= \begin{bmatrix}
		0 & 1 \\
		- \lambda & - \delta
	\end{bmatrix},
    \quad
	G = \begin{bmatrix} 0 \\ 1 \end{bmatrix},
    \quad\text{and}\quad
	C = \begin{bmatrix} 1 & 0\end{bmatrix}.
\end{equation*}
The parameters for the system and the disturbances for the construction of the output $y$ are set to 
\begin{equation*}
	\lambda = -1,~ \beta = 1,~ \delta =\tfrac{3}{10},~ v(t) = \gamma \cos(\omega t),~ \gamma = \tfrac{1}{2},~ \omega = \tfrac{12}{10}.
\end{equation*}
This setting is known to cause chaotic behavior \cite{JoSm07,BreKu21} and results in a clear discrepancy between the extended Kalman filter and the Mortensen observer \cite{BreKuSc23}.
The disturbance in the output is set to
$\mu(t) = \tfrac{1}{20} \sin(2 \pi t)$
and for the initial state and its disturbance we set 
$x_0 = \begin{bmatrix}
	0 & 0
\end{bmatrix}^\top$ and $\eta = \begin{bmatrix}
	-1.216 & 0.493
\end{bmatrix}^\top$.
Note that this combination of a chaotic system with challenging disturbances makes the state estimation an extremely difficult task. In fact, no observer without properly tuned weighting matrices can be expected to yield satisfactory reconstructions of the state. In our experiments we consider the time horizon $[0,T] = [0,10]$ and construct the higher order extended Kalman filters up to order eight. The $k$-th order observer is denoted by $\widehat{x}_k$ and the Mortensen observer constructed for comparison is referred to as $\widehat{x}_\mathrm{M}$. Since the output is linear, the second order observer $\widehat{x}_2$ coincides with the extended Kalman filter, cf.~\eqref{eq: ext Kalman}. Finally, the disturbed trajectory and the measured output are denoted as $x$ and $y$, respectively.
While the weighting matrices for the initial state and the process are set constant to $\Gamma = I_2$ and $R = 1$, we realize the observers for two different weightings of the output. 

The results are presented in \Cref{fig:duffing_observers}, see the left column for $Q = \tfrac{1}{2}$ and the right column for $Q = 2$. As announced, a clear difference is visible when comparing the extended Kalman filter $\widehat{x}_2$ with the Mortensen observer $\widehat{x}_\mathrm{M}$, both in the position and velocity and for $Q = \tfrac{1}{2} $ and $Q = 2 $. Note that for both tracking weights the Mortensen observer $\widehat{x}_\mathrm{M}$ is only displayed roughly up to $t = 6$. For larger times the gradient descent scheme used for the approximation of $\widehat{x}_\mathrm{M}$ fails. Comparing the results for the two tracking weights we find that the discrepancy between $\widehat{x}_2$ and $\widehat{x}_\mathrm{M}$ is less pronounced for the larger weighting of the output. 
We further point out that for $Q = \tfrac{1}{2} $ the third order observer $\widehat{x}_3$ does not exist on the full time interval but appears to exhibit a finite time blow-up around $t = 2$, cf.~\Cref{subf: duffing_pos}. Since no existence theory is available for the higher order observers, this is an expected risk and illustrates one of the challenges for future work. In contrast to that, for $Q = 2$ the trajectory $\widehat{x}_3$ exists for the full duration of the experiment.
In line with this we find that for $Q=2 $ the observers $\widehat{x}_k$ appear to be close to converging to $\widehat{x}_\mathrm{M}$ starting at $k = 4$. For $Q = \tfrac{1}{2}$ this effect is only visible starting at $k = 5$. This observation is additionally supported by the relative distance to the Mortensen observer
\begin{equation*}
    d_k(t) \coloneqq \frac{\Vert \widehat{x}_k(t) - \widehat{x}_\mathrm{M}(t) \Vert}{\Vert \widehat{x}_\mathrm{M}(t) \Vert},
    \qquad k = 2,\dots,8,
\end{equation*}
displayed in \Cref{subf: duffing_dist} and \Cref{subf: duffing_dist_a2}.
This example illustrates how higher order Kalman filters can successfully approximate the energy optimal observer. Keeping in mind that solving for $\widehat{x}_{\mathrm{M}}$ requires the value function, e.g., the solution of the HJB equation, the presented approach may offer a significant decrease in computational effort. 

%
%%%%%%%%%%%%%%%%%%%%%%%%%%%%%%%%%%%%%%%%%%%%%%%%%%%%%%%%%%%%%%
%%%%%%%%%%%%%%%%%Duffing observer plots%%%%%%%%%%%%%%%%%%%%%%%
%%%%%%%%%%%%%%%%%%%%%%%T=10;a=0.5%%%%%%%%%%%%%%%%%%%%%%%%%%%%%
%
\begin{figure}
	%%%%%%%%%%%%%%%%%%%%%%%%%%%%%%%%%%%%%%%%%%%%%%%%%%%%%%%%%%%
	%Position a = 0.5
	%%%%%%%%%%%%%%%%%%%%%%%%%%%%%%%%%%%%%%%%%%%%%%%%%%%%%%%%%%%
	\begin{subfigure}{0.49\textwidth}
        \includegraphics[scale = 0.95]{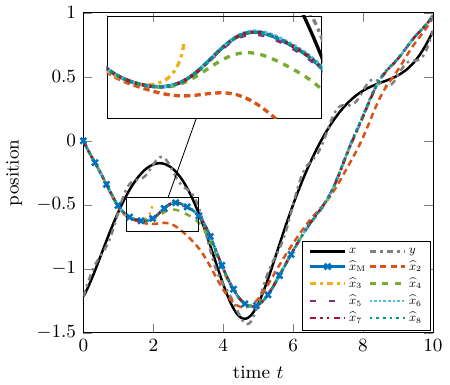}
		\caption{Positions for $Q = \tfrac{1}{2}$}
		\label{subf: duffing_pos}
	\end{subfigure}
    %%%%%%%%%%%%%%%%%%%%%%%%%%%%%%%%%%%%%%%%%%%%%%%%%%%%%%%%%%%
	%Position a = 2
	%%%%%%%%%%%%%%%%%%%%%%%%%%%%%%%%%%%%%%%%%%%%%%%%%%%%%%%%%%%
	\begin{subfigure}{0.49\textwidth}
        \includegraphics[scale = 0.95]{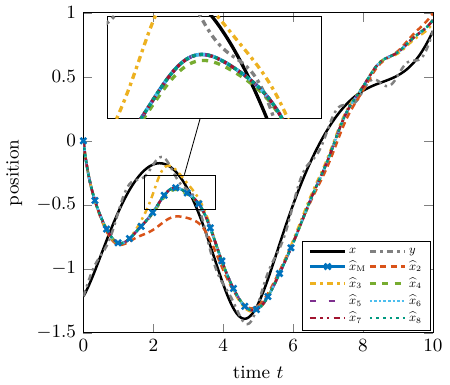}
		\caption{Positions for $Q = 2$}
		\label{subf: duffing_pos_a2}
	\end{subfigure}
	%%%%%%%%%%%%%%%%%%%%%%%%%%%%%%%%%%%%%%%%%%%%%%%%%%%%%%%%%%%
	%Velocity a = 0.5
	%%%%%%%%%%%%%%%%%%%%%%%%%%%%%%%%%%%%%%%%%%%%%%%%%%%%%%%%%%%
	\begin{subfigure}{0.49\textwidth}
        \includegraphics[scale = 0.95]{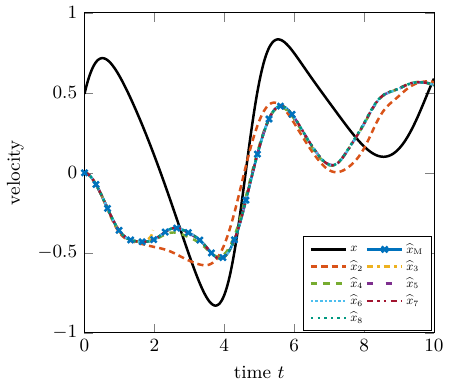}
		\caption{Velocities for $Q = \tfrac{1}{2}$}
		\label{subf: duffing_vel}
	\end{subfigure}
	%%%%%%%%%%%%%%%%%%%%%%%%%%%%%%%%%%%%%%%%%%%%%%%%%%%%%%%%%%%
	%Velocity a = 2
	%%%%%%%%%%%%%%%%%%%%%%%%%%%%%%%%%%%%%%%%%%%%%%%%%%%%%%%%%%%
	\begin{subfigure}{0.49\textwidth}
        \includegraphics[scale = 0.95]{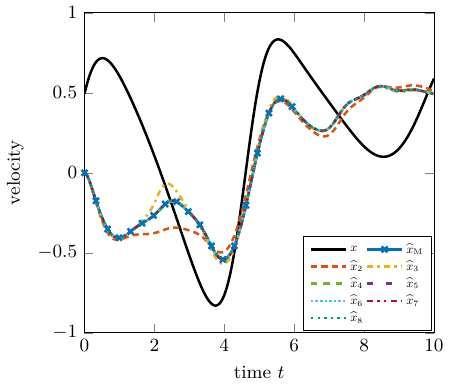}
		\caption{Velocities for $Q = 2$}
		\label{subf: duffing_vel_a2}
	\end{subfigure}
	%%%%%%%%%%%%%%%%%%%%%%%%%%%%%%%%%%%%%%%%%%%%%%%%%%%%%%%%%%%
	%Relative distance a = 0.5
	%%%%%%%%%%%%%%%%%%%%%%%%%%%%%%%%%%%%%%%%%%%%%%%%%%%%%%%%%%%
	\begin{subfigure}{0.49\textwidth}
        \includegraphics[scale = 0.95]{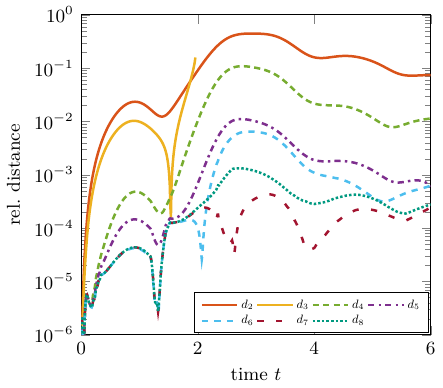}
		\caption{Relative distance to Mortensen for $Q = \tfrac{1}{2}$}
		\label{subf: duffing_dist}
	\end{subfigure}
	%%%%%%%%%%%%%%%%%%%%%%%%%%%%%%%%%%%%%%%%%%%%%%%%%%%%%%%%%%%
	%Relative distance a = 2
	%%%%%%%%%%%%%%%%%%%%%%%%%%%%%%%%%%%%%%%%%%%%%%%%%%%%%%%%%%%
	\begin{subfigure}{0.49\textwidth}
        \includegraphics[scale = 0.95]{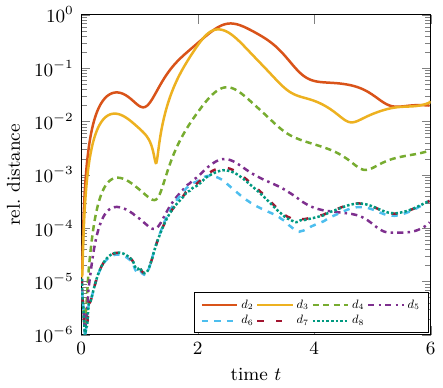}
		\caption{Relative distance to Mortensen for $Q = 2$}
		\label{subf: duffing_dist_a2}
	\end{subfigure}
    \caption{Duffing oscillator for the two tracking weights $Q = \tfrac{1}{2}$ and $Q = 2$}
    \label{fig:duffing_observers}
\end{figure}
%

%
%%%%%%%%%%%%%%%%%%%%%%%%%%%%%%%%%%%%%%%%%%%%%%%%%%%%%%%%%%%%%%%%%%%%%%%%%%%%%%%
%%%%%%%%%%%%%%%%%%%%%%%%%%%%%%%%%%%%%%%%%%%%%%%%%%%%%%%%%%%%%%%%%%%%%%%%%%%%%%%
\subsection{Semi-discretized cubic wave equation}
%%%%%%%%%%%%%%%%%%%%%%%%%%%%%%%%%%%%%%%%%%%%%%%%%%%%%%%%%%%%%%%%%%%%%%%%%%%%%%%
%%%%%%%%%%%%%%%%%%%%%%%%%%%%%%%%%%%%%%%%%%%%%%%%%%%%%%%%%%%%%%%%%%%%%%%%%%%%%%%
%

In the second experiment we consider the defocusing, cubic wave equation given by 
\begin{equation*}
\begin{aligned}
	\partial_{tt} {u}(t,x) + {u}^3(t,x) &= \Delta {u}(t,x) + \nu(t,x)
	~~~~~&&(t,x) \in (0,T) \times \Omega, \\
	{u}(0,x) &= u_1(x) + \eta_1(x)
	~~~~~&&x \in \Omega,\\
	\partial_t {u}(0,x) &= u_2(x) + \eta_2(x)
	~~~~~&&x \in \Omega,
\end{aligned}
\end{equation*}
with  spatial domain $\Omega = (0,1) \times (0,1) \subset \mathbb{R}^2$, given initial data $u_1 \in H^1_0(\Omega)$ and $u_2 \in L^2(\Omega)$, homogeneous Dirichlet boundary conditions, and unknown disturbances $\eta_1 \in H^1_0(\Omega)$, $\eta_2 \in L^2(\Omega)$, and $\nu \in L^2(0,T;L^2(\Omega))$. The output is modeled as 
\begin{equation*}
	y(t) = \mathcal{C} \begin{bmatrix} u(t) \\ \partial_t u(t) \end{bmatrix}
	+ \mu(t)
	~~~~~ t \in (0,T), 
\end{equation*}
with a bounded, linear operator $\mathcal{C} \colon H^1(\Omega) \times L^2(\Omega) \to \mathbb{R}^p$, $p \in \mathbb{N}$, and the observation error $\mu \in L^2(0,T;\mathbb{R}^p)$.
In \cite{Sch24} we studied the Mortensen observer for an equation of this type with a three-dimensional domain $\Omega$. Specifically, we reformulated the system in first order form using mild solutions, derived the associated observer equation and showed local well-posedness of the estimator.
For the sake of completeness we repeat the critical concepts here. The first order form reads
\begin{equation}\label{eq: waveNum}
	\begin{aligned}
		\frac{\mathrm{d}}{\mathrm{d}t} \begin{bmatrix} w_1(t) \\ w_2(t) \end{bmatrix}
		&=
		\mathcal{A}
		\begin{bmatrix} w_1(t) \\ w_2(t) \end{bmatrix}
		+
		\begin{bmatrix} 0 \\ - w_1^3(t) \end{bmatrix}
		+ \begin{bmatrix} 0 \\ \nu(t) \end{bmatrix}
		~~~~~ t \in (0,T),
		\quad
		\begin{bmatrix} w_1(0) \\ w_2(0) \end{bmatrix}
		= w_0 + \eta,\\
		y(t) &= \mathcal{C} \begin{bmatrix} w_1(t) \\ w_2(t) \end{bmatrix} + \mu(t) 
		~~~~~ t \in (0,T),
	\end{aligned}
\end{equation}
where $\mathcal{A} = \begin{bmatrix} 0 & \mathrm{Id} \\ \Delta & 0 \end{bmatrix}$ denotes the unbounded wave operator in the energy space $\mathcal{E} = H^1_0(\Omega) \times L^2(\Omega)$ and we set $w_0 = [u_1,u_2] \in \mathcal{E}$ and $\eta = [\eta_1,\eta_2] \in \mathcal{E}$.
The observation operator $\mathcal{C}$ used in this specific experiment is constructed as follows. 
We assume access to time-continuous measurements of the wave's displacement $w_1$ stemming from $p \in \mathbb{N}$ sensors placed in the domain $\Omega$ and denote the position of the $i$-th sensor by $(\nu_1^i,\nu_2^i) \in \Omega$. Since $w_1(t,\cdot) \in H^1_0(\Omega)$ might not be continuous, point evaluations are not available. Instead, we model each sensor to average the displacement on a square of length $2 \ell > 0$ around $(\nu_1^i,\nu_2^i)$ and arrive at the output operator 
\begin{equation*}
	\mathcal{C} \begin{bmatrix} w_1 \\ w_2 \end{bmatrix} 
	=
	\frac{1}{4\ell^2}
	\begin{bmatrix}
		\int_{S^1} 
		w_1(x) \, \mathrm{d} x 
		~\dots~
		\int_{S^p} 
		w_1(x) \, \mathrm{d} x
	\end{bmatrix}^\top,
\end{equation*}
where $S^i = (\nu_1^i - \ell,\nu_1^i + \ell) \times (\nu_2^i - \ell,\nu_2^i + \ell)$ denotes the square of length $2 \ell$ around the $i$-th sensor.

The extension of the Mortensen observer to the wave equation proposed in \cite{Sch24} is based on the value function $\mathcal{V}(t,\Xi)$ associated with the control problem
\begin{equation}\label{eq: OCPWaveNum}
	\begin{aligned}
		\min J(w,\nu;t,y_\mathrm{abs})
		&= \tfrac{1}{2} \Vert w(0) - w_0 \Vert_\mathcal{E}^2
		+ \tfrac{1}{2} \int_0^t \Vert \nu(s) \Vert_{L^2(\Omega)}^2
		+ \Vert y(s) - \mathcal{C} w(s) \Vert^2 \, \mathrm{d}s\\
		\text{subject to}~~~~~~~~~~&\\
		\frac{\mathrm{d}}{\mathrm{d}s} w(s)
		&=
		\mathcal{A}
		w(s)
		+
		\begin{bmatrix} 0 \\ - w_1^3(s) \end{bmatrix}
		+ \begin{bmatrix} 0 \\ \nu(s) \end{bmatrix}
		~~~~~ s \in (0,t),\\
		w(t)
		& = \Xi,
	\end{aligned}
\end{equation}
defined on infinite-dimensional Hilbert spaces. We note that the three terms in the cost function $J$ appear in their respective Hilbert space norms. This is to be understood as the infinite-dimensional extension of using Euclidean norms, i.e., setting the weights $\Gamma$, $R$, and $Q$ as identities.

%
%%%%%%%%%%%%%%%%%%%%%%%%%%%%%%%%%%%%%%%%%%%%%%%%%%%%
%%%%%%%%%%%%%%%%%%%%%%%%%%%%%%%%%%%%%%%%%%%%%%%%%%%%
\subsubsection{Discretization in space}\label{subs: discretization in space}
%%%%%%%%%%%%%%%%%%%%%%%%%%%%%%%%%%%%%%%%%%%%%%%%%%%%
%%%%%%%%%%%%%%%%%%%%%%%%%%%%%%%%%%%%%%%%%%%%%%%%%%%%
%
We briefly describe the spatial discretization based on spectral-like elements. For a more detailed presentation we refer to the thesis \cite{Schr25}.
For $i,j \in \mathbb{N}$ we define the basis function
\begin{equation*}
	\varphi_{i,j} \colon \Omega \rightarrow \mathbb{R},
	~~~~~ \varphi_{i,j}(x,y) = \sin(\pi i x ) \, \sin(\pi j y ),
\end{equation*}
and obtain the basis $\left\{ \varphi_{i,j} \mid i,j \in \mathbb{N} \right\}$ that is orthogonal both in $H^1_0(\Omega)$ and $L^2(\Omega)$.
For any time-dependent trajectory of states $w(t,x_1,x_2) = \begin{bmatrix} w_1(t,x_1,x_2) \\ w_2(t,x_1,x_2) \end{bmatrix}$ we have
\begin{equation}\label{eq: basisRep}
	w_1(t,x_1,x_2) = \sum_{i,j = 1}^\infty z_{1,i,j}(t) \varphi_{i,j}(x_1,x_2),
	~~~~~
	w_2(t,x_1,x_2) = \sum_{k = 1}^\infty z_{2,i,j}(t) \varphi_{i,j}(x_1,x_2),
\end{equation}
with appropriate, time-dependent coordinates $ z_{1,i,j}$ and $ z_{2,i,j}$. We obtain the spatial discretization by choosing an integer $K \geq 2$ and truncating the basis to
\begin{equation*}
	\mathcal{B} = \left\{ \varphi_{i,j} \colon i + j \leq K \right\},
\end{equation*}
containing $N=\frac{K(K-1)}{2}$ basis elements.
Inserting the approximation of $w_1$ and $w_2$ obtained by truncating \eqref{eq: basisRep} using the constructed basis (and analogous approximations of $\nu$, $w_0$, and $\eta$) into the control problem \eqref{eq: OCPWaveNum} yields a finite-dimensional approximation where the states are characterized by the coordinates $z \colon (0,T) \rightarrow \mathbb{R}^{2N},~ z(\cdot) = \begin{bmatrix} z_1(\cdot) \\ z_2(\cdot) \end{bmatrix}$.
By algebraic manipulations and a state space transformation one arrives at the equivalent formulation
\begin{equation}\label{eq: SemiDOCPTraf}
	\begin{aligned}
		\min \mathcal{J}(z,v;t,y) 
		&= \frac{1}{2} \Vert z(0) - \mathcal{T} z_0 \Vert^2
		+ \frac{1}{2} \int_0^t \Vert v(s) \Vert^2 + \Vert y(s) - C z(s) \Vert^2 \, \mathrm{d}s,\\
		\text{subject to}~~~~~~~~~& \\
		\dot{z}
		&= A z
		+ \begin{bmatrix} 0 \\ M^{-\tfrac{1}{2}} g(M^{-\tfrac{1}{2}}z_1) \end{bmatrix}
		+ B v,\\
		z(t) &= \mathcal{T} \xi,
	\end{aligned}
\end{equation}
where $z_0$, $\xi$, and $v$ correspond to the truncated coordinates of $w_0$, $\Xi$, and $\nu$, respectively. The system matrices are given as
\begin{equation}\label{eq:sysMatWave}
	A = \begin{bmatrix} 0 & S^{\tfrac{1}{2}} M^{-\tfrac{1}{2}} 
		\\ - M^{-\tfrac{1}{2}} S^{\tfrac{1}{2}} & 0 \end{bmatrix},
	~~~~~~
	B = \begin{bmatrix} 0 \\ \mathrm{Id} \end{bmatrix},
	~~~~~~
	C = \Tilde{C} \, T^{-1},
\end{equation}
where $\Tilde{C}= \begin{bmatrix} \Tilde{C}_1 & 0 \end{bmatrix} \in \mathbb{R}^{p,2N}$ is defined via 
\begin{equation*}
	\Tilde{C}_1 =
	\frac{1}{4\ell^2}
	\begin{bmatrix} \int_{S^i} \varphi_k (x) \, \mathrm{d}x \end{bmatrix}_{i=1,k=1}^{i=p,k=N},
\end{equation*}
and
\begin{equation*}
	M = \left[ (\varphi_i,\varphi_j)_{L^2(\Omega)} \right]_{i,j=1}^N,
	\quad 
	S = \left[ (\varphi_i,\varphi_j)_{H^1_0(\Omega)} \right]_{i,j=1}^N
	\quad 
    \text{and}
    \quad 
	\mathcal{T} = \begin{bmatrix} S^{\tfrac{1}{2}} & 0 \\ 0 & M^{\tfrac{1}{2}} \end{bmatrix}.
\end{equation*}
The nonlinearity is
\begin{equation*}
	g(z) 
	= g\left(\begin{bmatrix} z_1 \\ z_2 \end{bmatrix} \right) 
	= \left[ \left( (\sum_{k=1}^N z_{1,k} \varphi_k )^3 , \varphi_i \right)_{L^2(\Omega)} \right]_{i=1}^N.
\end{equation*}
%

%
%%%%%%%%%%%%%%%%%%%%%%%%%%%%%%%%%%%%%%%%%%%%%%%%%%%%
%%%%%%%%%%%%%%%%%%%%%%%%%%%%%%%%%%%%%%%%%%%%%%%%%%%%
\subsubsection{Realization of the estimators}\label{subsubsec:WaveMort}
%%%%%%%%%%%%%%%%%%%%%%%%%%%%%%%%%%%%%%%%%%%%%%%%%%%%
%%%%%%%%%%%%%%%%%%%%%%%%%%%%%%%%%%%%%%%%%%%%%%%%%%%%
%
We realize the approximations of the Mortensen observer for the semi-discretized system associated with the control problem \eqref{eq: SemiDOCPTraf} of dimension $2N$. The system matrices $(A,B,C)$ are given in \eqref{eq:sysMatWave} and the cost functional $J$, see \eqref{eq: OCPWaveNum}, in combination with our chosen discretization scheme implies that the weights $\Gamma$, $R$, and $Q$ are given by identities of appropriate sizes, cf.~\eqref{eq: SemiDOCPTraf}. The higher order extended Kalman filters are then realized as described in the beginning of \Cref{sec:Num}. 

As mentioned above the dimension of the system prohibits a realization of the Mortensen observer by means of the approach that was applied for the first example. Instead, we construct a data-driven approximation $\mathcal{V}_p$ of the associated value function using a slightly adjusted version of the scheme proposed in \cite{BreKuSc23}. The estimators are then realized either by solving the observer equation using $\nabla^2 \mathcal{V}_p^{-1}$ or by minimizing $\mathcal{V}_p(t,\cdot)$. 
We begin by pointing out the minor adjustments. First, the gradient descent scheme for solving open-loop control problems to generate the data set is altered. More precisely, we omit the line search applied in \cite[Alg.~3.1]{BreKuSc23} and instead set the step size directly via Barzilai-Borwein.
The second adjustment is concerned with fitting a polynomial to the data set. Denoting the evaluation of the polynomial basis functions in the sampling points by $H$, the vector containing the data points by $V$, and the size of the polynomial basis by $N_p$ it reads
\begin{equation*}
    \min_{\theta \in \mathbb{R}^{N_p}}
	\Vert H \, \theta - V \Vert_2^2.
\end{equation*}
Here the large number of samples and basis functions required for a sufficiently accurate approximation of the value function lead to a dense matrix $H$ of enormous dimension. Hence, due to memory limitations, 
a direct solve is no longer feasible. 
Instead we solve
\begin{equation*}
	\min_{\theta \in \mathbb{R}^{N_p}}
	\Vert H^\top H \, \theta - H^\top V \Vert_2^2
\end{equation*}
using the \Matlab, backslash routine, i.e., $\theta = (H^\top H) \backslash (H^\top V) $, where $H^\top H$ and $H^\top V$ are constructed in a block-wise, iterative fashion, sidestepping the memory issue.

We turn to the specifics of the implementation and note that the gradient descent scheme terminates when reaching a relative tolerance of $10^{-6}$.
The associated state equations and adjoint equations are solved by an application of a BDF4 scheme using 501 time discretization points. The nonlinearities are treated by a Newton scheme with absolute tolerance $10^{-12}$. Since the Riccati equations yielding evaluations of the Hessians depend on the previously realized optimal states and adjoints, we integrate the DREs using a BDF4 scheme using the same time nodes. The implicit time steps are realized by the \Matlab~ routine icare. Only for the four initial time steps a Newton scheme is employed, where the absolute tolerance is set to $10^{-8}$. The first initial control for the gradient descent is set to zero and the second one is obtained by performing the first step of the gradient descent with fixed step size $10^{-4}$.
With an approximation of the value function at hand the observer equation is integrated using a BDF4 scheme with $10^3$ discretization points and a Newton scheme scheme with absolute tolerance $10^{-8}$. Also, the approximated value function is minimized using a gradient descent scheme with relative and absolute tolerance $10^{-6}$.

%
%%%%%%%%%%%%%%%%%%%%%%%%%%%%%%%%%%%%%%%%%%%%%%%%%%%%
%%%%%%%%%%%%%%%%%%%%%%%%%%%%%%%%%%%%%%%%%%%%%%%%%%%%
\subsubsection{Results}
%%%%%%%%%%%%%%%%%%%%%%%%%%%%%%%%%%%%%%%%%%%%%%%%%%%%
%%%%%%%%%%%%%%%%%%%%%%%%%%%%%%%%%%%%%%%%%%%%%%%%%%%%
%
For our experiment we set the time horizon to $T = 2$ and for the spatial discretization of the PDE we set $K=4$ implying the use of $N= 6$ basis elements. Hence, the resulting system is of dimension $12$. We use $p=16 $ sensors that are placed on an equidistant grid of the domain and are characterized via the length $l = 0.01$, cf.~\Cref{fig: sensors} for an illustration in gray.  
Note that the placement of the sensors is suboptimal and could be improved upon, see for example \cite{WuJaEl16} and the references therein. 
The modeled initial state is set to 
\begin{equation*}
\begin{aligned}
    w_0(x_1,x_2) &= \begin{bmatrix} 8 \, \sin(\pi \, x_1) \, \sin(\pi \, x_2)\\0 \end{bmatrix}
	 + \begin{bmatrix} 4 \, \sin(\pi \, x_1) \, \sin(2 \pi \, x_2)\\0 \end{bmatrix}\\
	 &+ \begin{bmatrix} 2 \, \sin(\pi \, x_1) \, \sin(3 \pi \, x_2)\\0 \end{bmatrix}
	 + \begin{bmatrix} 2 \, \sin(2 \pi \, x_1) \, \sin(\pi \, x_2)\\0 \end{bmatrix}\\
	 &+ \begin{bmatrix} \, \sin(3 \pi \, x_1) \, \sin(\pi \, x_2)\\0 \end{bmatrix}
	 ,
\end{aligned}
\end{equation*}
leading to $z_0 = \begin{bmatrix} 8 & 4 & 2 & 2 & 0 & 1 & 0 & 0 & 0 & 0 & 0 & 0 \end{bmatrix}^\top$. 
The output $y$ is constructed by solving the system with the disturbances
\begin{equation*}
\begin{aligned}
	 \eta(x_1,x_2) &= \begin{bmatrix} - \sin(\pi \, x_1) \, \sin(\pi \, x_2)\\0 \end{bmatrix},\\
	 v(t,x_1,x_2) &= \tfrac{1}{10} \vert \sin(\tfrac{1}{2} \, \pi t) \vert \sin(\pi \, x_1) \sin(2 \pi x_2),\\
	 \mu_k(t) & = \frac{1}{20} \left(-2 + \frac{4(k-1)}{15} \right) \sin(5 \, \pi t),
	 ~~~~~~ k = 1,\dots,16.
\end{aligned}
\end{equation*}
For illustration, the disturbed initial displacement is plotted in \Cref{fig: sensors}.

For the data-driven approximation of the value function we use 200 time and 20 spatial samples, requiring a total of $4 \times 10^3$ value function evaluations. The sampling points are drawn from a neighborhood of the extended Kalman filter with radius $r_{k,i} = \max \left( 1, \tfrac{1}{10} \vert \widehat{z}_\mathrm{K} (t_k)_i \vert \right)$,  where $\widehat{z}_\mathrm{K}$ denotes the coordinates of the EKF. 
For the minimization based approximation of the estimator we build a polynomial basis using $24$ degrees in time and a hyperbolic cross with index $9$ for the spatial indices. The data points include only values and gradients with weight one for both, Hessians are omitted. The resulting estimator is denoted by $\widehat{x}_\mathrm{min}$.
On the other hand, the realization obtained by solving the observer equation is denoted by $\widehat{x}_\mathrm{eq}$ and for the associated value function approximation we use a polynomial basis with 29 degrees in time and hyperbolic cross index 9. The data includes the values with weight 0.01 and Hessians with weight 1. 

As mentioned above, the higher order extended Kalman filters are realized exactly as for the first example. We solve for the estimator $\widehat{x}_k$ of order $k$ for $k = 3, 4, 5$ and compare the results to the extended Kalman filter $\widehat{x}_\mathrm{K} = \widehat{x}_2$ and the two data-driven approximations $\widehat{x}_\mathrm{min}$ and $\widehat{x}_\mathrm{eq}$. 
Since all estimators aim at approximating the minimizer of the value function $\mathcal{V}$, we assess their approximation quality using two metrics: the evaluation of their respective energies, and the normed gradients of the value function. The energies are presented in \Cref{subf: wave_en}. We observe that the extended Kalman filter and the HOEKF3 display noticeably larger energies than the remaining estimators, which appear to perform similarly well. This indicates that for this example the commonly used extended Kalman filter displays noticeable differences from the energy optimal reconstruction. However, the higher order extended Kalman filters of order four and five seem to perform well in terms of the energy. In particular, their quality matches the data-driven results at a fraction of the computational effort. In fact, the energies of $\widehat{x}_4$, $\widehat{x}_5$, $\widehat{x}_\mathrm{min}$ and $\widehat{x}_\mathrm{eq}$ are virtually the same, with only a slight difference in favor of the HOEKF.
The normed gradients presented in \Cref{subf: wave_grad} offer a more nuanced insight. One finds that the gradients of $\widehat{x}_\mathrm{min}$ and $\widehat{x}_\mathrm{eq}$ are of magnitude $10^{-2}$ at most times, suggesting that they offer a reasonably good approximation of the minimizer. However, looking at $\widehat{x}_4$ and $\widehat{x}_5$ we find that they yield noticeably smaller gradients with $\Vert \nabla \mathcal{V}(t,\widehat{x}_5(t)) \Vert$ staying below $3 \times 10^{-3}$ at all times. 
We conclude that, for this example, the higher order Kalman filter is a viable alternative for the approximation of the minimum energy estimator that outperforms the previously available data-driven approach both in approximation quality and in computational efficiency. 

%%%%%%%%%%%%%%%%%%%%%%%%%%%%%%%%%%%%%%%%%%%%%%%%%%%%%%%%%%%%%%
%%%%%%%%%%%%%%%%%Wave Plots%%%%%%%%%%%%%%%%%%%%%%%%%%%%%%%%%%%
%%%%%%%%%%%%%%%%%%%%%%%%%%%%%%%%%%%%%%%%%%%%%%%%%%%%%%%%%%%%%%
\begin{figure}
    \center
	%%%%%%%%%%%%%%%%%%%%%%%%%%%%%%%%%%%%%%%%%%%%%%%%%%%%%%%%%%%
	%Energy
	%%%%%%%%%%%%%%%%%%%%%%%%%%%%%%%%%%%%%%%%%%%%%%%%%%%%%%%%%%%
	\begin{subfigure}{0.99\textwidth}
        \center
		\includegraphics[scale = 0.95]{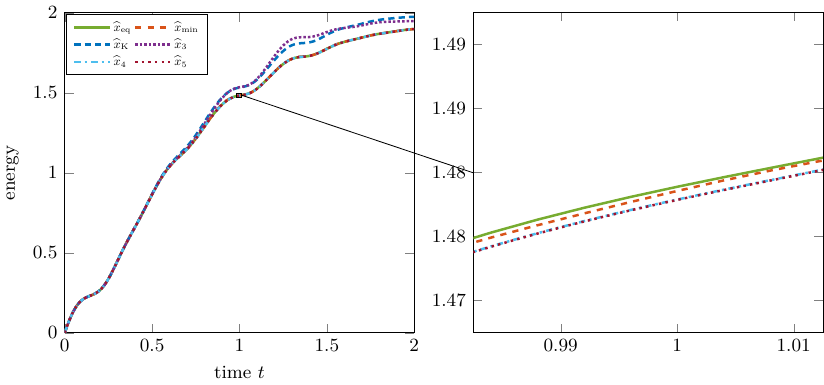}
		\caption{Energy along the estimators}
		\label{subf: wave_en}
	\end{subfigure}
	%%%%%%%%%%%%%%%%%%%%%%%%%%%%%%%%%%%%%%%%%%%%%%%%%%%%%%%%%%%
	%Normed gradient
	%%%%%%%%%%%%%%%%%%%%%%%%%%%%%%%%%%%%%%%%%%%%%%%%%%%%%%%%%%%
	\begin{subfigure}{0.49\textwidth}
		\includegraphics[scale = 0.95]{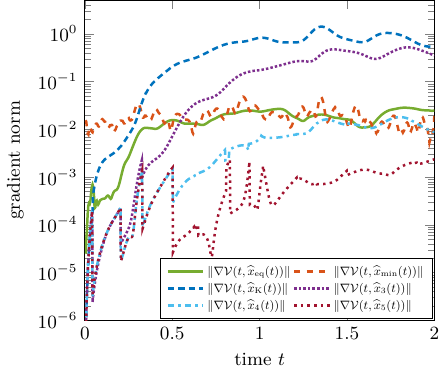}
		\caption{Gradient norms of the value function}
		\label{subf: wave_grad}
	\end{subfigure}
    \hfill
	%%%%%%%%%%%%%%%%%%%%%%%%%%%%%%%%%%%%%%%%%%%%%%%%%%%%%%%%%%%
	%Sensors and initial state
	%%%%%%%%%%%%%%%%%%%%%%%%%%%%%%%%%%%%%%%%%%%%%%%%%%%%%%%%%%%
    \begin{subfigure}{0.49\textwidth}
    		\includegraphics[scale = 0.55]{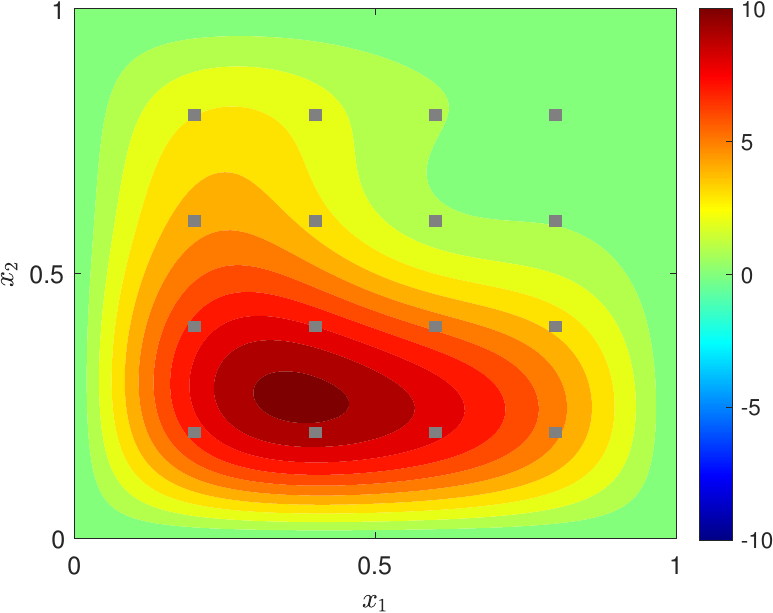}
    		\caption{Disturbed initial value and sensor placement}
    		\label{fig: sensors}
    \end{subfigure}
    \caption{Comparison of EKF, HOEKFs, and the data-driven approximations}
\end{figure}
%

%
%%%%%%%%%%%%%%%%%%%%%%%%%%%%%%%%%%%%%%%%%%%%%%%%%%%%%%%%%%%%%%%%%%%%%%%%%%%%%%%
%%%%%%%%%%%%%%%%%%%%%%%%%%%%%%%%%%%%%%%%%%%%%%%%%%%%%%%%%%%%%%%%%%%%%%%%%%%%%%%
\section{Conclusion \& Outlook}
%%%%%%%%%%%%%%%%%%%%%%%%%%%%%%%%%%%%%%%%%%%%%%%%%%%%%%%%%%%%%%%%%%%%%%%%%%%%%%%
%%%%%%%%%%%%%%%%%%%%%%%%%%%%%%%%%%%%%%%%%%%%%%%%%%%%%%%%%%%%%%%%%%%%%%%%%%%%%%%
%
By means of successive differentiations of the HJB equation, we derived a constitutive set of coupled vector, matrix and tensor differential equations generalizing the classical extended Kalman filter. From a theoretical perspective, the approach is justified by assuming that a higher order derivative of the minimum energy value function is negligible in the vicinity of the unknown Mortensen observer trajectory. The provided numerical experiment shows a considerable improvement of our approach over the extended Kalman filter in the sense of (numerically) converging to the exact Mortensen observer when higher order polynomials are considered.

While the method appears promising both in terms of required numerical effort as well as approximately solving the minimum energy estimation problem, many direct and related research questions remain open. First of all, we have not discussed existence and uniqueness of a solution to \cref{EQ:HigherOrderObserver}. Even if this can be shown, the results should only be expected to be of a local nature as already our numerical experiments indicate a strong dependence of the involved parameters and the degree of the considered polynomials. On the other hand, a theoretical study of \cref{EQ:HigherOrderObserver} could further open up a way to analytically explain the frequently reported success of the extended Kalman filter. From a numerical perspective, the next steps should evolve around a more efficient implementation for solving \cref{EQ:HigherOrderObserver}, a natural starting point being the recently developed tensor methods for Taylor series based feedback stabilization from \cite{Kraetal24,CorK25}.
%
%%%%%%%%%%%%%%%%%%%%%%%%%%%%%%%%%%%%%%%%%%%%%%%%%%%%%%%%%%%%%%%%%%%%%%%%%%%%%%%
%%%%%%%%%%%%%%%%%%%%%%%%%%%%%%%%%%%%%%%%%%%%%%%%%%%%%%%%%%%%%%%%%%%%%%%%%%%%%%%
%Literature
%%%%%%%%%%%%%%%%%%%%%%%%%%%%%%%%%%%%%%%%%%%%%%%%%%%%%%%%%%%%%%%%%%%%%%%%%%%%%%%
%%%%%%%%%%%%%%%%%%%%%%%%%%%%%%%%%%%%%%%%%%%%%%%%%%%%%%%%%%%%%%%%%%%%%%%%%%%%%%%
%

%%-----------------------------
\bibliographystyle{siam}
\bibliography{references} 

\begin{thebibliography}{10}

\bibitem{AguK14}
{\sc C.~O. Aguilar and A.~J. Krener}, {\em Numerical solutions to the {B}ellman
  equation of optimal control}, J.~Optim.~Theory Appl., 160 (2014),
  pp.~527--552.
\newblock https://doi.org/10.1007/s10957-013-0403-8.

\bibitem{Alb61}
{\sc E.~G. Al'brekht}, {\em On the optimal stabilization of nonlinear systems},
  J.~Appl.~Math.~Mech., 25 (1961), p.~1254.
\newblock https://doi.org/10.1016/0021-8928(61)90005-3.

\bibitem{BaKo25}
{\sc B.~W. Bader, T.~G. Kolda, et~al.}, {\em Tensor toolbox for matlab}, Dec.
  2025.
\newblock www.tensortoolbox.org.

\bibitem{Bel61}
{\sc R.~Bellman}, {\em Adaptive Control Processes: A Guided Tour}, Princeton
  University Press, 1961.

\bibitem{BorZ21}
{\sc J.~Borggaard and L.~Zietsman}, {\em On approximating polynomial-quadratic
  regulator problems}, IFAC-PapersOnLine, 54 (2021), pp.~329--334.
\newblock https://doi.org/10.1016/j.ifacol.2021.06.090.

\bibitem{BreKu21}
{\sc T.~Breiten and K.~Kunisch}, {\em Neural network based nonlinear
  observers}, Syst.~Control Lett., 148 (2021).
\newblock https://doi.org/10.1016/j.sysconle.2020.104829.

\bibitem{BreKP18}
{\sc T.~Breiten, K.~Kunisch, and L.~Pfeiffer}, {\em Numerical study of
  polynomial feedback laws for a bilinear control problem}, Math.~Control
  Relat.~Fields, 8 (2018), pp.~557--582.
\newblock https://doi.org/10.3934/mcrf.2018023.

\bibitem{BreKP19}
\leavevmode\vrule height 2pt depth -1.6pt width 23pt, {\em Taylor expansions of
  the value function associated with a bilinear optimal control problem},
  Ann.~Inst.~Henri Poincare (C) Anal.~Non Lineaire, 36 (2019), pp.~1361--1399.
\newblock https://doi.org/10.1016/j.anihpc.2019.01.001.

\bibitem{BreKuSc23}
{\sc T.~Breiten, K.~Kunisch, and J.~Schr\"{o}der}, {\em Numerical realization
  of the {M}ortensen observer via a {H}essian-augmented polynomial
  approximation of the value function}, SIAM J.~Sci.~Comput., 47 (2025),
  pp.~A181--A206.
\newblock https://doi.org/10.1137/23M1613773.

\bibitem{BrRaSc26}
{\sc T.~Breiten, J.~Ramme, and J.~Schröder}, {\em Code for the paper
  "{A}pproximations of the {M}ortensen observer using higher order extended
  {K}alman filters"}, Apr. 2026.
\newblock https://doi.org/10.5281/zenodo.19855072.

\bibitem{BreS24}
{\sc T.~Breiten and J.~Schr\"oder}, {\em Local well-posedness of the
  {M}ortensen observer}, ESAIM - Control Optim.~Calc.~Var., 30 (2024).
\newblock https://doi.org/10.1051/cocv/2024046.

\bibitem{ChaEtAl23}
{\sc L.-P. Chaintron, {\'A}.~M. González, L.~Mertz, and P.~Moireau}, {\em
  {M}ortensen observer for a class of variational inequalities – lost
  equivalence with stochastic filtering approaches}, ESAIM, Proc.~surv., 73
  (2023), pp.~130--157.
\newblock https://doi.org/10.1051/proc/202373130.

\bibitem{ChauMeMoZi25}
{\sc L.-P. Chaintron, L.~Mertz, P.~Moireau, and H.~Zidani}, {\em Constrained
  non-linear estimation and links with stochastic filtering}, 2025.
\newblock https://doi.org/10.48550/arXiv.2502.01200.

\bibitem{CorK25}
{\sc N.~A. Corbin and B.~Kramer}, {\em Scalable computation of $\mathcal
  {H}_\infty$ energy functions for polynomial control-affine systems}, IEEE
  Trans.~Autom.~Control, 70 (2025), pp.~3088--3100.
\newblock https://doi.org/10.1109/TAC.2024.3494472.

\bibitem{Fle97}
{\sc W.~H. Fleming}, {\em Deterministic nonlinear filtering},
  Ann.~Sc.~Norm.~Super.~Pisa, Cl. Sci., 25 (1997), pp.~435--454.

\bibitem{FleMcE00}
{\sc W.~H. Fleming and W.~M. McEneaney}, {\em A max-plus-based algorithm for a
  {H}amilton--{J}acobi--{B}ellman equation of nonlinear filtering}, SIAM
  J.~Control Optim., 38 (2000), pp.~683--710.
\newblock https://doi.org/10.1137/S0363012998332433.

\bibitem{Hac19}
{\sc W.~Hackbusch}, {\em Tensor spaces and numerical tensor calculus}, vol.~56
  of Springer Series in Computational Mathematics, Springer, Cham, second~ed.,
  2019.
\newblock https://doi.org/10.1007/978-3-030-35554-8.

\bibitem{Har06}
{\sc M.~Hardy}, {\em Combinatorics of partial derivatives}, Electron.~J.~Comb.,
  13 (2006).

\bibitem{Hij80}
{\sc O.~Hijab}, {\em Minimum energy estimation}, ph.{D}.~dissertation,
  University of California, Berkeley, 1980.

\bibitem{Hij82}
{\sc O.~Hijab}, {\em Asymptotic nonlinear filtering and large deviations}, in
  Advances in Filtering and Optimal Stochastic Control, Berlin Heidelberg,
  1982, Springer, pp.~170--176.
\newblock https://doi.org/10.1007/BFb0004536.

\bibitem{JoSm07}
{\sc D.~Jordan and P.~Smith}, {\em Nonlinear Ordinary Differential Equations:
  An Introduction for Scientists and Engineers}, vol.~10, Oxford University
  Press on Demand, fourth~ed., 2007.
\newblock https://doi.org/10.1093/oso/9780199208241.001.0001.

\bibitem{Kraetal24}
{\sc B.~Kramer, S.~Gugercin, J.~Borggaard, and L.~Balicki}, {\em Scalable
  computation of energy functions for nonlinear balanced truncation},
  Comput.~Methods Appl.~Mech.~Eng, 427 (2024), p.~117011.
\newblock https://doi.org/10.1016/j.cma.2024.117011.

\bibitem{KreAH13}
{\sc A.~Krener, C.~Aguilar, and T.~Hunt}, {\em Series solutions of {HJB}
  equations}, in Mathematical System Theory -- Festschrift in Honor of Uwe
  Helmke on the Occasion of his Sixtieth Birthday, K.~H\"uper and J.~Trumpf,
  eds., CreateSpace, 2013, pp.~247--260.

\bibitem{Kre03}
{\sc A.~J. Krener}, {\em The convergence of the minimum energy estimator}, in
  New Trends in Nonlinear Dynamics and Control and their Applications,
  Springer, Berlin, 2004, pp.~187--208.

\bibitem{Kre15}
\leavevmode\vrule height 2pt depth -1.6pt width 23pt, {\em Minimum energy
  estimation and moving horizon estimation}, in 2015 54th IEEE Conference on
  Decision and Control (CDC), 2015, pp.~4952--4957.
\newblock https://doi.org/10.1109/CDC.2015.7402993.

\bibitem{Kre18}
\leavevmode\vrule height 2pt depth -1.6pt width 23pt, {\em Minimum Energy
  Estimation Applied to the Lorenz Attractor}, Springer International
  Publishing, Cham, 2018, pp.~165--182.
\newblock https://doi.org/10.1007/978-3-030-01959-4.

\bibitem{Luk69}
{\sc D.~Lukes}, {\em Optimal regulation of nonlinear dynamical systems}, SIAM
  J.~on Contr., 7 (1969), pp.~75--100.
\newblock https://doi.org/10.1137/0307007.

\bibitem{Moi18}
{\sc P.~Moireau}, {\em A discrete-time optimal filtering approach for
  non-linear systems as a stable discretization of the {M}ortensen observer},
  ESAIM - Control Optim.~Calc.~Var., 24 (2018), pp.~1815--1847.
\newblock https://doi.org/10.1051/cocv/2017077.

\bibitem{Mor68}
{\sc R.~E. Mortensen}, {\em Maximum-likelihood recursive nonlinear filtering},
  J.~Optim.~Theory Appl., 2 (1968), pp.~386--394.
\newblock https://doi.org/10.1007/BF00925744.

\bibitem{Sal22}
{\sc L.~Sallandt}, {\em Computing high-dimensional value functions of optimal
  feedback control problems using the tensor-train format}, doctoral thesis,
  Technische Universität Berlin, 2022.
\newblock https://doi.org/10.14279/depositonce-12786.

\bibitem{Schr25}
{\sc J.~Schr\"oder}, {\em Energy-optimal state reconstruction for deterministic
  nonlinear dynamical systems}, doctoral thesis, Technische Universität
  Berlin, 2025.
\newblock https://doi.org/10.14279/depositonce-23323.

\bibitem{Sch24}
{\sc J.~Schröder}, {\em Local well-posedness of the minimum energy estimator
  for a defocusing cubic wave equation}, J.~Differ.~Equ., 435 (2025),
  p.~113258.
\newblock https://doi.org/10.1016/j.jde.2025.113258.

\bibitem{SchB24a}
{\sc J.~Schröder and T.~Breiten}, {\em {Code for the thesis ``Energy-Optimal
  State Reconstruction for Deterministic Nonlinear Dynamical Systems''}}, 2024.
\newblock https://doi.org/10.5281/zenodo.14394797.

\bibitem{Wil04}
{\sc J.~C. Willems}, {\em Deterministic least squares filtering}, J.~Econom.,
  118 (2004), pp.~341--373.
\newblock https://doi.org/10.1016/S0304-4076(03)00146-5.

\bibitem{Wol19}
{\sc A.~Wolf}, {\em Low rank tensor decompositions for high dimensional data
  approximation, recovery and prediction}, doctoral thesis, Technische
  Universität Berlin, 2019.
\newblock https://doi.org/10.14279/depositonce-8109.

\bibitem{WuJaEl16}
{\sc X.~Wu, B.~Jacob, and H.~Elbern}, {\em Optimal control and observation
  locations for time-varying systems on a finite-time horizon}, SIAM J.~Control
  Optim., 54 (2016), pp.~291--316.
\newblock https://doi.org/10.1137/15M1014759.

\end{thebibliography}
%%-----------------------------

\end{document}